\documentclass[12pt]{amsart}
\usepackage[margin=1in]{geometry}
\usepackage{libertine}
\usepackage[vvarbb,cmintegrals,libertine]{newtxmath}
\usepackage{esint}
\usepackage{graphicx}
\usepackage{setspace}
\usepackage{amsthm,amscd,centernot}
\usepackage{hyperref}
\usepackage{color}
\usepackage{booktabs}
\usepackage{tabularx}
\usepackage{enumitem}
\usepackage[retainorgcmds]{IEEEtrantools}
\usepackage[notref,notcite,final]{showkeys}
\usepackage[final]{pdfpages}
\usepackage{fancyhdr}
\usepackage{ytableau}
\usepackage{stmaryrd}
\usepackage{cleveref}
\usepackage[T1]{fontenc}
\usepackage{verbatim}
\usepackage{bbm}

\makeatletter
\newcommand{\fixed@sra}{$\vrule height 2\fontdimen22\textfont2 width 0pt\rightarrow$}
\newcommand{\diagarrow}{%
\mathrel{\text{\rotatebox[origin=c]{\numexpr45}{\fixed@sra}}}
}
\makeatother

\usepackage{tikz}
\usepackage{tikz-cd}
\usetikzlibrary{cd}

\newtheorem{theorem}{Theorem}[section]
\newtheorem*{theorem*}{Theorem}
\newtheorem{lemma}[theorem]{Lemma}
\newtheorem*{lemma*}{Lemma}

\newtheorem{corollary}[theorem]{Corollary}

\newtheorem{Definition}[theorem]{Definition}
\newenvironment{definition}
{\begin{Definition}\rm}{\end{Definition}}
\newtheorem{Example}[theorem]{Example}

\newtheorem{Question}[theorem]{Question}

\theoremstyle{definition}
\newtheorem{remark}[theorem]{\textbf{Remark}}

\newcommand{\E}{\mathbb{E}}
\newcommand{\F}{\mathbb{F}}
\newcommand{\N}{\mathbb{N}}
\newcommand{\Q}{\mathbb{Q}}
\newcommand{\R}{\mathbb{R}}

\newcommand{\Z}{\mathbb{Z}}

\newcommand{\calc}{\mathcal{C}}
\newcommand{\cald}{\mathcal{D}}
\newcommand{\cale}{\mathcal{E}}
\newcommand{\calf}{\mathcal{F}}
\newcommand{\calg}{\mathcal{G}}
\newcommand{\calh}{\mathcal{H}}

\newcommand{\calm}{\mathcal{M}}

\newcommand{\ev}{\mathfrak{e}}

\DeclareMathOperator{\im}{\mathrm{im}}

\DeclareMathOperator{\ann}{\mathrm{Ann}}

\DeclareMathOperator{\Hom}{\mathrm{Hom}}
\DeclareMathOperator{\Aut}{\mathrm{Aut}}
\DeclareMathOperator{\Sym}{\mathrm{Sym}}

\DeclareMathOperator{\Sur}{\mathrm{Sur}}

\DeclareMathOperator{\Prob}{\mathbb{P}}
\DeclareMathOperator{\col}{\mathrm{col}}

\newcommand{\ip}[1]{\langle#1\rangle}

\newcommand{\inverse}{^{-1}}

\newcommand{\bpm}[1]{\begin{pmatrix}#1\end{pmatrix}}

\newcommand{\vphi}{\varphi}

\newcommand{\cok}{\mathrm{cok}}
\newcommand{\coker}{\mathrm{coker}}
\newcommand{\rank}{\mathrm{rank}}
\newcommand{\hG}{\widehat{G}}

\newcommand{\h}[1]{\widehat{#1}}
\newcommand{\p}{\ip{\, ,}}
\newcommand{\phG}{\ip{\, ,}_{\widehat{G}}}

\newcommand{\hcok}{\mathrm{c}\widehat{\mathrm{ok}}}
\newcommand{\htcok}{\mathrm{tc}\widehat{\mathrm{ok}}}
\newcommand{\hF}{\widehat{F}}
\newcommand{\Div}{\mathrm{Div}}
\newcommand{\Prin}{\mathrm{Prin}}
\newcommand{\ord}{\mathrm{ord}}
\newcommand{\di}{\mathrm{div}}
\newcommand{\finab}{\mathrm{FinAb}}
\newcommand{\tcok}{\mathrm{tcok}}

\setlist[enumerate,1]{leftmargin=1.8em, label=\textup{(}\arabic*\textup{)}}

\begin{document}
\title{The Distribution of Sandpile Groups of Random Graphs with their Pairings}

\author{Eliot Hodges}
\address{Department of Mathematics, Harvard University, Cambridge, Massachusetts 02138}
\email{eliothodges@college.harvard.edu}
\begin{abstract}
	We determine the distribution of the sandpile group (also known as the Jacobian) of the Erd\H{o}s-R\'{e}nyi random graph $ G(n,q) $ along with its canonical duality pairing as $ n $ tends to infinity, fully resolving a conjecture from 2015 due to Clancy, Leake, and Payne and generalizing the result by Wood on the groups. In particular, we show that a finite abelian $ p $-group $ G $ equipped with a perfect symmetric pairing $ \delta $ appears as the Sylow $ p $-part of the sandpile group and its pairing with frequency inversely proportional to $ |G||\Aut(G,\delta)| $, where $ \Aut(G,\delta) $ is the set of automorphisms of $ G $ preserving the pairing $ \delta $. While this distribution is related to the Cohen-Lenstra distribution, the two distributions are not the same on account of the additional algebraic data of the pairing. The proof utilizes the moment method: we first compute a complete set of moments for our random variable (the average number of epimorphisms from our random object to a fixed object in the category of interest) and then show the moments determine the distribution. To obtain the moments, we prove a universality result for the moments of cokernels of random symmetric integral matrices whose dual groups are equipped with symmetric pairings that is strong enough to handle both the dependence in the diagonal entries and the additional data of the pairing. We then apply results due to Sawin and Wood to show that these moments determine a unique distribution. 
\end{abstract}
\maketitle
\section{Introduction}\label{Introduction}
To any graph $ \Gamma $ we may associate a natural abelian group $ S_\Gamma $, which is commonly referred to as the sandpile group, the Jacobian, the critical group, or the Picard group of the graph. The sandpile group is an important object of interest in several subjects, including but not limited to statistical mechanics, combinatorics, and arithmetic geometry (hence its several names; see \cite{NW11} for an overview of these various connections). If $ \Gamma $ is connected, then $ S_\Gamma $ is finite, and $ |S_\Gamma| $ is the number of spanning trees of $ \Gamma $. The sandpile group of a graph comes naturally equipped with an additional important piece of algebraic data: a canonical \emph{duality pairing} (i.e., a perfect symmetric pairing) $ S_\Gamma\times S_\Gamma\to\Q/\Z $, which was first introduced by Bosch and Lorenzini in \cite{BL02}. In \cite{Lor08}, Lorenzini conjectured the frequency that $ S_\Gamma $ is cyclic and posed several other questions about statistics of the distribution of sandpile groups of random graphs. Initially, some suspected that the distribution of sandpile groups of random graphs obeyed a Cohen-Lenstra \cite{CL84} heuristic---a natural first guess for the distribution of a random finite abelian group. However, Clancy, Leake, and Payne collected empirical data in \cite{CLP} suggesting that this was not the case. In particular, Clancy, Leake, and Payne noticed that a finite abelian group with duality pairing $ (G,\p_G) $ seems to appear as the sandpile group and pairing of a random graph with frequency proportional to \[\frac{1}{|G||\Aut(G,\p_G)|},\] where $ \Aut(G,\p_G) $ denotes the automorphisms of $ G $ preserving the pairing. In other words, they observed that the natural duality pairing on $ S_\Gamma $ seemed to be affecting its distribution. In \cite{CLK+}, Clancy, Kaplan, Leake, Payne, and Wood proved an analogous result where the sandpile group is replaced by the cokernel of a Haar-random symmetric matrix with entries in $ \Z_p $ (such groups also have duality pairings). Later, in \cite{Wood}, Wood proved the distribution of sandpile groups of Erd\H{o}s-R\'{e}nyi random graphs, almost fully resolving the conjecture of Clancy, Leake, and Payne. However, Wood's approach was unable to determine the distribution of the pairings, and the conjecture regarding the distribution of the pairings has since remained open (this conjecture is restated as Open Problem 3.13 in \cite{WoodICM}). In this paper, we extend Wood's result to show that the sandpile group of an Erd\H{o}s-R\'{e}nyi random graph with its pairing is distributed as Clancy, Leake, and Payne conjectured. 

Let $ \Gamma\in G(n,q) $ be an Erd\H{o}s-R\'{e}nyi random graph on $ n $ vertices with edges occurring independently with probability $ q $, where $ 0<q<1 $. Let $ S_\Gamma $ denote the sandpile group of $ \Gamma $, and for a prime $ p $ let $ S_{\Gamma,p} $ denote the Sylow $ p $-subgroup of $ S_\Gamma $. Also, let $ \p_{S_\Gamma} $ denote the duality pairing on the torsion subgroup of $ S_\Gamma $ and $ \p_{S_\Gamma,p} $ its restriction to $ S_{\Gamma,p} $.\footnote{As we will see in \Cref{sec: Pairings}, it is actually the torsion subgroup of the sandpile group that naturally carries a duality pairing. When $ \Gamma $ is connected, the sandpile group is finite and the torsion subgroup is all of $ S_\Gamma $. However, since the graphs in question are Erd\H{o}s-R\'{e}nyi random, we must be more careful, since Erd\H{o}s-R\'{e}nyi random graphs are not necessarily connected (though they are asymptotically almost surely connected).} It follows from Wood's work in \cite{Wood} that, for a given finite abelian group $ G $, we have $ \lim_{n\to\infty}\Prob(S_\Gamma\simeq G)=0 $. Recalling that a finite abelian group can be decomposed into a direct sum of its Sylow $ p $-subgroups, we instead ask: for $ G $ a finite abelian $ p $-group with duality pairing $ \p_G $, what is the probability as $ n\to\infty $ that $ (S_{\Gamma,p},\p_{S_\Gamma,p})\simeq(G,\p_G) $? \begin{theorem}\label{thm: distribution of gps and pairings}
	Let $ p $ be a prime and $ G $ a finite abelian $ p $-group with duality pairing $ \p_G $. For a random graph $ \Gamma\in G(n,q) $, let $ S $ denote its sandpile group and $ \p_S $ the duality pairing on its torsion subgroup. Then \begin{equation}\label{distribution of sandpile group and pairing}
		\lim_{n\to\infty}\Prob((S_{p},\p_{S,p})\simeq(G,\p_G))=\frac{\prod_{k=1}^\infty(1-p^{1-2k})}{|G||\Aut(G,\p_G)|}.
	\end{equation}
\end{theorem}

\Cref{thm: distribution of gps and pairings} has several interpretations: we may regard it both as a result in combinatorics on random graphs and as a result in arithmetic statistics (there are analogies between sandpile groups and Jacobians of curves over finite fields \cite{BL02,Lopez}). In addition to proving \Cref{thm: distribution of gps and pairings}, we show a stronger result---\Cref{distribution random matrices}---along the way. \Cref{distribution random matrices} is a universality theorem on cokernels of random matrices with duality pairings which implies the analogous result due to Wood (Theorem 1.3 of \cite{Wood}) by an application of the Orbit-Stabilizer Theorem. 

We remark that the numerator on the right-hand side of \eqref{distribution of sandpile group and pairing} does not depend on $ G $ and that this product is essentially a normalization constant. Moreover, $ \prod_{k=1}^\infty(1-p^{1-2k}) $ increases with $ p $, and for $ p>1000 $ we have $ \prod_{k=1}^\infty(1-p^{1-2k})>0.999 $. Note that the right-hand side of \eqref{distribution of sandpile group and pairing} is independent of the edge probability $ q $. Further, in \Cref{distribution of sandpile groups multiple primes}, we are able to show a much stronger result than \Cref{thm: distribution of gps and pairings}: for any finite set of primes, we give the asymptotic probabilities of particular Sylow subgroups with their pairings at these primes.

The proof of \Cref{thm: distribution of gps and pairings} uses a result due to Sawin and Wood on the moment problem for random objects in diamond categories, where a moment of a random object is defined to be the average number of epimorphisms to a fixed object in the category. In their paper \cite{SW22}, Sawin and Wood show that as long as the moments do not grow too fast, a unique measure with these moments exists, and a sequence of measures whose moments approach these particular moments in the limit converge weakly to the aforementioned measure.\footnote{In our case, the moments grow too quickly to use classical probabilistic methods to show the moments determine a unique distribution.} Because $ \p_{S_\Gamma} $ is a duality pairing, there is a corresponding duality pairing on $ \h{S_\Gamma}=\Hom(S_\Gamma,\Q/\Z) $, the Pontryagin dual group of $ S_\Gamma $, which we denote by $ \p_{\h{S_\Gamma}} $. To prove \Cref{thm: distribution of gps and pairings}, we first compute a complete set of moments for $ (S_\Gamma,\p_{\h{S_\Gamma}}) $.

It is not immediately clear which category we should be working in or why we are considering pairings on $ \h{S_\Gamma} $ rather than $ S_\Gamma $. We define a category $ \calc $ whose objects are finite abelian groups whose dual groups are equipped with symmetric pairings (i.e., the objects of $ \calc $ are finite abelian groups $ G $ such that $ \hG=\Hom(G,\Q/\Z) $ has a symmetric pairing $ \phG:\hG\otimes\hG\to\Q/\Z $). Note that if $ (A,\p_{\h{A}}) $ and $ (B,\p_{\h{B}}) $ are objects in $ \calc $, and if $ F:A\to B $ is a group homomorphism, then the transpose of $ F $, the map $ F^t:\h{B}\to\h{A} $, pushes the pairing on $ \h{A} $ forward to a pairing on $ \h{B} $ by precomposition with $ F^t\otimes F^t $. A morphism between $ (A,\p_{\h{A}}) $ and $ (B,\p_{\h{B}}) $ in $ \calc $ is a group homomorphism between $ A $ and $ B $ that pushes the pairing $ \p_{\h{A}} $ on $ \h{A} $ forward to the pairing $ \p_{\h{B}} $ on $ \h{B} $. The fact that $ F^t $ pushes the pairing on $ \h{A} $ forward to a pairing on $ \h{B} $ is the primary reason for considering pairings on the dual groups. Suppressing the symbols for the pairings on $ \h{A} $ and $ \h{B} $ from our notation, let $ \Sur^*(A,B) $ denote the set of surjective homomorphisms from $ A $ to $ B $ taking $ \p_{\h{A}} $ to $ \p_{\h{B}} $.

The following gives the moments of random sandpile groups with their pairings.
\begin{theorem}\label{sandpile moments intro}
	Let $ G $ be a finite abelian group, and let $ \phG $ be a symmetric pairing on its dual, $ \hG=\Hom(G,\Q/\Z) $. Then for a random graph $ \Gamma\in G(n,q) $ with $ (S_\Gamma,\p_{\h{S_\Gamma}}) $ its sandpile group and symmetric pairing on $ \h{S_\Gamma} $, we have \[\lim_{n\to\infty}\E(\#\Sur^*(S_\Gamma,G))=\frac{1}{|G|}.\]
\end{theorem}

We call $ \E(\#\Sur^*(S_\Gamma,G)) $ the $ (G,\phG) $-moment of $ (S_\Gamma,\p_{\h{S_\Gamma}}) $. The error term for \Cref{sandpile moments intro} decreases exponentially in $ n $ (see \Cref{sandpile moments}).

\subsection{Sandpile groups and their pairings}
We refer the reader to ``What is\ldots\ a sandpile?'' \cite{LP} for an overview of sandpile groups. The article \cite{NW11} also gives an outline of the ways in which sandpile groups are connected to various disparate subjects. While there are many equivalent definitions of the sandpile group of a graph $ \Gamma $, perhaps the easiest is the following. Suppose $ \Gamma $ is a simple graph with $ n $ vertices labeled by the first $ n $ integers. The \emph{Laplacian} of $ \Gamma $, denoted $ L_\Gamma $, is the $ n\times n $ matrix whose $ (i,j) $ entry is given by \[(L_\Gamma)_{ij}=\begin{cases}
	1 & \mathrm{if\ }i\neq j\ \mathrm{and\ }\{i,j\}\ \mathrm{is\ an\ edge\ of\ }\Gamma;\\
	0 &\mathrm{if\ }i\neq j\ \mathrm{and\ }\{i,j\}\ \mathrm{is\ not\ an\ edge\ of\ }\Gamma;\\
	-\deg(i) & \mathrm{if\ }i=j.
\end{cases}\] The \emph{reduced Laplacian} of $ \Gamma $, denoted $ \Delta_\Gamma $, is the $ (n-1)\times(n-1) $ matrix given by deleting the final row and column from $ L_\Gamma $. We define the sandpile group of $ \Gamma $, denoted $ S_\Gamma $, to be the cokernel of the reduced Laplacian $ \Delta_\Gamma $ of $ \Gamma $. We remark here that the sandpile group does not depend on the row or column removed from $ L_\Gamma $---the cokernels of the matrices given by deleting any row or column of $ L_\Gamma $ are all isomorphic (see \cite{Biggs,CP}). From this definition, it follows immediately that if $ \Gamma $ is connected, then $ |S_\Gamma|=|\det(\Delta_\Gamma)| $; Kirchhoff's matrix tree theorem implies that $ |S_\Gamma| $ is the number of spanning trees of $ \Gamma $. Alternatively, if $ Z\subset\Z^n $ denotes the subspace of vectors whose coordinates sum to 0, then it is equivalent to define $ S_\Gamma $ as $ S_\Gamma=Z/L_\Gamma\Z $. 

Sandpile groups were first studied in 1988 by statistical physicists interested in the dynamics of a sandpile \cite{BTW}; thus originating the name ``sandpile.'' The sandpile model is an example of a chip-firing game (see \cite{BLS}): given a graph $ \Gamma $, we imagine that on each vertex is some number of chips; if a vertex has at least as many chips as its degree, the chips topple, distributing a chip to each neighboring vertex. Supposing that chips fall randomly onto the vertices of $ \Gamma $, toppling whenever they can, the sandpile group of $ \Gamma $ is given by the recurrent states of the resulting Markov chain. For more in depth treatments, we refer the reader to \cite{Dhar,Gab93a,Gab93b,Biggs99}. There is also a nice connection between sandpile groups and the Tutte polynomial of a graph; see \cite{Lopez,Gab93a,Gab93b}. An overview of these connections can be found in \cite{HLM+}.

The sandpile group of a graph has also been referred to in the literature as the Jacobian (also the Picard group or critical group) of the graph. This moniker originates from a beautiful analogy between Riemann surfaces and graphs, in which the sandpile group corresponds to the Jacobian of a Riemann surface \cite{Biggs,BLHPdN}. This analogy is not purely philosophical and is spelled out in detail in \cite{Lor89,BL02}, where the group of components of the N\'{e}ron model of the Jacobian of a curve over a local field is realized as the sandpile group of a graph. Moreover, the cardinality of the sandpile group occurs in the ``analytic class number formula'' for graphs \cite{HST}, and there also exist analogs of the Riemann-Roch and Riemann-Hurwitz formulas for Jacobians of graphs \cite{BN07,BN09}.

Sandpile groups of connected graphs are canonically equipped with a perfect, symmetric bilinear pairing (see \cite{Lor00,BL02}), which is given by any generalized inverse of the Laplacian (see \cite{Shokrieh}). If the graph is not connected, then only the torsion subgroup of the sandpile group carries such a duality pairing. We will show in \Cref{sec: Pairings} that the duality pairing on $ \h{S_\Gamma} $ induced by the pairing on $ S_\Gamma $ can be described directly using the Laplacian of the graph. In the Riemann surface analogy, the pairing on $ S_\Gamma $ has been considered as an analog of the Weil pairing \cite{Shokrieh} and also agrees with a pairing defined by Grothendieck \cite{Gr1,GR} giving an obstruction to extending the Poincar\'{e} bundle to the product of an abelian variety and its dual \cite{BL02}.

\subsection{The Cohen-Lenstra Heuristics}
For a family of random finite abelian groups, a natural first guess for how the groups should be distributed is the Cohen-Lenstra  heuristics \cite{CL84}, which conjecture the distribution of ideal class groups of quadratic number fields. Ideal class groups are finite abelian groups measuring the failure of unique factorization in rings of integers of number fields (e.g., $ \Z[\sqrt{-5}] $). The underlying philosophy behind the Cohen-Lenstra heuristics is that objects ``in nature'' often occur with frequency that is inversely proportional to their number of automorphisms. Hence, in the case of finite abelian groups, given a random finite abelian group $ X $ and a fixed finite abelian group $ G $, \emph{a priori} we expect $ \Prob(X\simeq G) $ to be inversely proportional to $ |\Aut(G)| $.

However, Wood showed in \cite{Wood} that sandpile groups of random graphs do not satisfy a Cohen-Lenstra heuristic. Instead, Wood proved that they obey a ``Cohen-Lenstra-like'' distribution, in which the canonical duality pairing on the group influences the way in which they are distributed. Notably, Wood's result is what we obtain by summing the conjectural distribution of random sandpile groups with their pairings due to Clancy, Leake, and Payne \cite{CLP} over all possible pairings on a given group. Similarly, Tate-Shafarevich groups of elliptic curves are a class of conjecturally finite and abelian groups that, if finite, also come equipped with a nondegenerate, bilinear pairing. Unlike sandpile groups, the pairing on Tate-Shafarevich groups is alternating rather than symmetric. Delaunay developed heuristics \cite{Del01,BKL} similar to those of Clancy, Leake, and Payne in which the automorphisms of $ G $ preserving the pairing take the place of $ \Aut(G) $. 

While the distribution of sandpile groups of random graphs with their pairings is intriguing in its own right, our methods may give insight into how to approach other similar problems in both random group theory and arithmetic statistics. Recently, Lipnowski, Sawin, and Tsimerman showed in \cite{LST20} that, in some cases, class groups of quadratic number fields come equipped with a bilinear structure closely related to the Weil pairing on abelian varieties and the aforementioned Cassels-Tate pairing on Tate-Shafarevich groups. It may be possible that the methods developed in this paper to deal with the pairing on the sandpile group can be used to determine the distribution of other families of random groups with pairings.
\subsection{Connections to random matrices}
The Laplacian is a matrix with integral entries, so if $ \Gamma $ is an Erd\H{o}s-R\'{e}nyi random graph, we inevitably find ourselves working in the realm of random matrix theory, a field in which significant effort has been devoted to finding (asymptotically) universal statistics of random matrices. In other words, it is often the case that some statistics are asymptotically equal for a broad class of random matrices, which often contain matrices from especially nice (e.g., symmetric) distributions for which we can actually explicitly compute these statistics. For example, a considerable amount of effort has been dedicated to finding such universality results for the eigenvalue distributions of random matrices: see \cite{Meh67,Girko04,Bai97,BS10,PZ10,TV06,TV07}.

In our scenario, we would like to find the distribution of the Sylow $ p $-subgroups of the cokernels of random integral matrices with their natural duality pairings on the torsion subgroups of these cokernels. In \Cref{sec: Pairings}, we will recall a theorem of Bosch and Lorenzini stating that the torsion subgroup of the cokernel of any integral symmetric matrix is naturally equipped with a duality pairing. For a random symmetric matrix with $ p $-adic integral entries, drawn with respect to Haar measure on $ \Z_p $, Clancy, Kaplan, Leake, Payne, and Wood showed in \cite{CLK+} that the cokernel of this matrix with its duality pairing are distributed as in \Cref{thm: distribution of gps and pairings}. These Haar-random matrices should be viewed as the matrices that are ``nicely distributed''---because Haar measure on $ \Z_p $ is well-understood, we can actually compute the distribution of the cokernels with their pairings. A similar computation is done in \cite{FW} for the cokernels (sans pairings) of random (not necessarily symmetric) $ p $-adic integral matrices and in \cite{BKL} for Haar-random alternating matrices.

Because we can compute the aforementioned distribution in the nice Haar case, a natural next line of inquiry is to see whether these asymptotic cokernel and pairing distributions are {universal}---whether other large families of random symmetric matrices have the same cokernel and pairing distributions. In particular, \Cref{distribution random matrices} is a strong universality result for distributions of cokernels of random symmetric integral matrices with their pairings. The proof of \Cref{distribution random matrices} is essentially the same as \Cref{thm: distribution of gps and pairings}. 
\begin{theorem}\label{distribution random matrices}
	Let $ 0<\alpha<1 $ be a real number, and let $ p $ be a prime. Let $ M $ be a symmetric random matrix in $ M_n(\Z) $ with independent entries $ m_{ij} $ for $ 1\leq i\leq j\leq n $. Suppose also that for any $ t\in\Z/p\Z $ the probability $ \Prob(m_{ij}\equiv t\ \mathrm{mod\ }p)\leq 1-\alpha $. Let $ \col(M) $ denote the column space of $ M $ in $ \Z^n $, let $ \cok(M)_p $ denote the Sylow $ p $-subgroup of $ \cok(M)=\Z^n/\col(M) $, and let $ \p_{\tcok(M)} $ denote the canonical duality pairing on the torsion subgroup of $ \cok(M) $ (see \Cref{sec: Pairings}). For any finite abelian $ p $-group $ G $ equipped with a duality pairing $ \p_G $, we have \[\lim_{n\to\infty}\Prob((\cok(M)_p,\p_{\tcok(M),p})\simeq(G,\p_G))=\frac{\prod_{k\geq1}(1-p^{1-2k})}{|G||\Aut(G,\p_G)|},\] where $ \p_{\tcok(M),p} $ denotes the restriction of $ \p_{\tcok(M)} $ to $ \cok(M)_p $. 
\end{theorem}

As in the case with the sandpile group, we are able to show a stronger result than \Cref{distribution random matrices}---for any finite set of primes, we give the asymptotic probabilities of particular Sylow subgroups with their pairings at these primes.

Recently, there has been significant effort devoted to finding universal distributions of various classes of random groups. In 2019, Wood showed that random integral matrices with independent entries are distributed according to a Cohen-Lenstra heuristic \cite{Wood19}, and in \cite{Eric}, Yan extended this result to random matrices with independent entries taking values in Dedekind domains with finite quotients. In 2020, M\'{e}z\'{a}ros gave the distribution of the sandpile groups of random regular graphs in \cite{Meszaros}. Koplewitz studied the $ p $-rank distribution of sandpile groups of random bipartite graphs in \cite{Koplewitz}, and Bhargava, DePascale, and Koenig considered the directed bipartite case in \cite{BDK}. Nguyen and Wood wrote several papers \cite{NW22,NW22nt} in which they study several other related problems on the distribution of cokernels of certain classes of random matrices. In particular, they found in \cite{NW22nt} the asymptotic probability that the sandpile group of an Erd\H{o}s-R\'{e}nyi random graph is cyclic, answering a long-open conjecture due to Lorenzini \cite{Lor08}. Lee found the distribution of cokernels of random Hermitian matrices with entries in the ring of integers of a quadratic extension of $ \Q_p $ \cite{Lee} and proved results of similar flavor in \cite{Lee2,Lee3}. In 2022, Nguyen and Van Peski found in \cite{NVP} the distribution of the cokernel of the product of random integral square matrices, and in 2023 Cheong and Yu proved a similar result in which they found the distribution of the cokernel of a polynomial evaluated at a random $ p $-adic square matrix \cite{CY}. While much of the work in this area has been devoted to studying universal distributions of abelian groups, there has also been interest in the nonabelian case. For example, Gorokhovsky proved analogous universality results for random quotients of the free group on $ n $ generators in \cite{Gor}, and Lee studied the distributions of similar random nonabelian groups in \cite{Lee3}.

\subsection{Our methods} The overall structure of our argument is based off of the seminal methods developed by Wood in \cite{Wood}. For a more detailed outline of the overall procedure, we refer the reader to the introduction of \cite{Wood} and instead focus on the new techniques developed to deal with the pairing. The most important of these techniques is our usage of duality. As mentioned previously, we will primarily be working with finitely generated abelian groups with symmetric pairings attached to their \emph{dual groups}. We will convert results about pairings on duals into statements about pairings on the groups themselves in \Cref{sec: Comparison to Haar}. In \Cref{sec: Pairings}, we attach a natural symmetric pairing to the Pontryagin dual of the cokernel of a symmetric integral matrix and show that, when the matrix is $ L_\Gamma $, the pairing agrees with the duality pairing on $ \h{S_\Gamma} $. To show \Cref{sandpile moments intro}, we compute the moments in a far more general setting where a random matrix as in \Cref{distribution random matrices} replaces the Laplacian of a random graph (see \Cref{random matrix moment}). Then, using results due to Sawin and Wood in \cite{SW22}, we prove (in \Cref{moments determine distribution}) that these moments in fact determine a unique distribution. Thus, we may use the nice case of cokernels of uniform random symmetric matrices over $ \Z/a\Z $ with their pairings, along with our universality results, to prove the distribution in the more general case (the distribution in the nice case was computed in \cite{CLK+}). 

While the method of proof is inspired by \cite{Wood}, the extra data of the pairing complicates some of the proofs significantly. Indeed, it is not immediately clear that we should consider pairings on the duals of our groups or what the moments of the groups with pairings should be. Thus, to compute the averages $ \Sur^*(S_\Gamma,G) $, we first translate the event that the pairing on $ \h{S_\Gamma} $ pushes forward onto the pairing on $ \hG $ into a piece of algebraic data which is easier to work with. In particular, if $ F:S_\Gamma\to G $ is a surjection, the event that $ F^t $ takes the pairing on $ \h{S_\Gamma} $ to the pairing on $ \hG $ occurs if and only if $ FL_\Gamma F^t $ is equal to some fixed symmetric homomorphism from $ \hG\to G $. Then since we are trying to compute the average number of surjections from $ S_\Gamma\to G $ satisfying this property, and since the sandpile group is $ Z/L_\Gamma Z $, it suffices to determine the probability that a surjection $ F:Z\to G $ descends to $ S_\Gamma $ and pushes the pairing on $ \h{S_\Gamma} $ forward to the pairing on $ \hG $. Equivalently, we determine the probability that $ FL_\Gamma=0 $ and that $ FL_\Gamma F^t $ is equal to some fixed symmetric element of $ \Hom(\hG,G) $. 

We regard this as a system of linear equations in the entries of $ L_\Gamma $, and we note that there are on the order of $ n $ equations in $ {{n}\choose{2}} $ variables. While in \cite{Wood} it was enough to check whether $ FL_\Gamma=0 $, the pairing data makes this system particularly complex not only because the equations are more complicated, but because the equations given by $ FL_\Gamma F^t $ are not even well-defined unless $ FL_\Gamma=0 $. To work around this inherent dependence between our equations, we choose a lift of $ F:Z\to G $ to a larger group $ H $ and modify our equations slightly so that we may work exclusively in this new group. Denote the lift by $ \calf $. Although this construction seems artificial, it turns out that all of our arguments depend only on $ F $ itself rather than the specific choice of lift. 

Our system of equations is parameterized by $ \Hom(Z,\hG)\times\Sym^2\h{H}$. Some pairs $ (C,D)\in\Hom(Z,\hG)\times\Sym^2\h{H}$, give equations in which all of the coefficients are 0; we refer to these as \emph{special} pairs. The set of special pairs $ (C,D) $ depends implicitly on $ F $ and the chosen lift $ \calf $. Ultimately, we are interested in the structural properties of $ \calf $ and $ C $ that influence the number of nonzero coefficients (see \Cref{sec: Moments II}). We adapt two concepts from Wood's work in \cite{Wood}, \emph{depth} and \emph{robustness}, to describe this structure. Finally, we count the number of special $ (C,D) $ for each $ F $. The special $ (C,D) $ give the main term in \Cref{sandpile moments} (the limit in \Cref{sandpile moments intro}), while the remaining cases contribute to the error term, which we bound by adapting techniques from \cite{Wood}. 

Lastly, because the diagonal entries of $ L_\Gamma $ depend on the off-diagonal ones, we run the program described above for a matrix with independent diagonal entries. Then we enlarge $ F $ by taking its direct sum with another map that detects whether the column sums of a matrix are 0. This allows us to condition on what we require of the diagonal. However, in order to apply our methods to the enlarged group we must also extend the pairing on $ \hG $ to the bigger group and sum over all possible extensions.

\subsection{Outline of the paper}
In \Cref{sec: Pairings}, for a symmetric integral matrix $ M $, we define a symmetric pairing on $ \hcok(M) $ and show that this induces a duality pairing on the torsion subgroup of $ \cok(M) $. The induced duality pairing is equal to the duality pairing on the torsion part of $ \cok(M) $ introduced by Bosch and Lorenzini in \cite{BL02}. \Cref{sandpile moments intro} (along with an analog of this theorem for cokernels of random symmetric matrices with their pairings) is proved in \Cref{sec: Moments I,sec: Moments II,sec: Moments III,sec: Moments IV,sec: Moments V}. In \Cref{moments determine distribution}, we state several results from \cite{SW22} and use these to prove that the moments of \Cref{sandpile moments intro} in fact determine the distribution in question. Finally, in \Cref{sec: Comparison to Haar}, we compare to the case of uniform random symmetric matrices over $ \Z/a\Z $; \Cref{thm: distribution of gps and pairings} follows from \Cref{distribution of sandpile groups multiple primes}.

\section{Background}\label{Background}
\subsection{Finite abelian groups} Let $ p $ be a prime. A finite abelian $ p $-group is isomorphic to \[\bigoplus_{i=1}^r\Z/p^{\lambda_i}\Z\] for some positive integers $ \lambda_1\geq\lambda_2\geq\cdots\geq\lambda_r $. If $ \lambda $ is the partition given by the $ \lambda_i $'s, we call $ \lambda $ the \emph{type} of our finite abelian $ p $-group. When $ p $ is understood, let $ G_\lambda $ denote the $ p $-group of type $ \lambda $. 

We may view any two abelian groups $ G $ and $ H $ as $ \Z $-modules, and we may consider their tensor product $ G\otimes_\Z H $. When $ G=\bigoplus_i\Z/a_i\Z $ is finite and $ H=\Z/a\Z $, \[G\otimes_\Z H=\left(\bigoplus_i\Z/a_i\Z\right)\otimes_\Z\Z/a\Z=\bigoplus_i\Z/(a,a_i)\Z,\] where $ (a,a_i) $ denotes the greatest common divisor of $ a $ and $ a_i $. The fact that tensor products distribute over direct sums in conjunction with the classification of finite abelian groups allows us to compute $ G\otimes H $ for arbitrary finite abelian groups $ G $ and $ H $. In particular, if $ G_p $ denotes the Sylow $ p $-subgroup of $ G $, then $ G\otimes G\simeq\bigoplus_pG_p\otimes G_p. $ Moreover, if $ G_p $ is of type $ \lambda $ and is generated by $ g_i $ with relations $ p^{\lambda_i}g_i=0 $, then $ G_p\otimes G_p $ is generated as an abelian group by $ g_i\otimes g_j $ with relations $ p^{\min(\lambda_i,\lambda_j)}g_i\otimes g_j=0 $. Thus, \[G_p\otimes G_p\simeq\bigoplus_i\left(\Z/p^{\lambda_i}\Z\right)^{2i-1}.\]

Now, $ G\otimes G $ comes equipped with a natural involution taking $ g\otimes h\mapsto h\otimes g $. Denote the subgroup of elements fixed under this involution by $ \Sym_2 G $. For $ G_p $ as above, we have that $ \Sym_2G\simeq\bigoplus_p\Sym_2G_p $ and that $ \Sym_2G_p $ is the abelian group generated by the elements $ g_i\otimes g_i $ and $ g_i\otimes g_j+g_j\otimes g_i $ for $ i< j $ with relations $ p^{\lambda_j}g_i\otimes g_j=p^{\lambda_j}g_j\otimes g_i=0 $ for each $ i<j $. Thus, $ |\Sym_2G_p|=p^{\sum_ii\lambda_i} $. 

The \emph{exterior power} $ \wedge^2G $ is the quotient of $ G\otimes G $ by the subgroup generated by elements of the form $ g\otimes g $. For $ G_p $ as above, we have that $ \wedge^2G\simeq\bigoplus_p\wedge^2G_p $ and that $ \wedge^2 G_p $ is generated as an abelian group by $ g_i\wedge g_j $ for $ i<j $ with relations $ p^{\lambda_j}g_i\wedge g_j=0 $ for each $ i<j $. Thus, \[\wedge^2G_p\simeq\bigoplus_i\left(\Z/p^{\lambda_i}\Z\right)^{i-1}.\]

Likewise, the \emph{symmetric power} $ \Sym^2G $ is the quotient of $ G\otimes G $ by the subgroup generated by elements of the form $ g\otimes h-h\otimes g $. For $ G_p $ as above, we have that $ \Sym^2G\simeq\bigoplus_p\Sym^2G_p $ and that $ \Sym^2 G_p $ is generated as an abelian group by $ g_ig_j $ for $ i\leq j $ with relations $ p^{\lambda_j}g_ig_j=0 $ for $ i\leq j $. Therefore, \[\Sym^2G_p\simeq\bigoplus_i\left(\Z/p^{\lambda_i}\Z\right)^i.\]

The \emph{exponent} of a finite abelian group is the smallest positive integer $ e $ such that $ eG=0 $. If $ R=\Z/a\Z $, then a finite abelian group $ G $ with exponent dividing $ a $ is also an $ R $-module. If $ H $ is another such group, then $ \Hom_R(G,H)=\Hom_\Z(G,H) $ (i.e., the $ R $-module homomorphisms $ G\to H $ are the same as the group homomorphisms $ G\to H $).

For any finite abelian group $ G $, its \emph{Pontryagin dual} is the group $ \Hom_\Z(G,\Q/\Z) $, which we denote by $ \hG $. 

\subsection{Pairings}\label{subsec: pairings}
Let $ G $ be an abelian group and consider a map $ \phi:G\times G\to\Q/\Z $. We say that $ \phi $ is \emph{symmetric} if $ \phi(g,h)=\phi(h,g) $ and \emph{bilinear} if $ \phi(g_1+g_2,h)=\phi(g_1,h)+\phi(g_2,h) $ and $ \phi(g,h_1+h_2)=\phi(g,h_1)+\phi(g,h_2) $. If $ \phi $ is bilinear, then it is called a \emph{pairing} on $ G $. Equivalently, a pairing on $ G $ can be thought of as a homomorphism $ G\otimes G\to\Q/\Z $. Given a pairing $ \phi $ and generators $ g_1,\ldots,g_r $ of $ G $, note that $\phi$ is entirely determined by the values $ \phi(g_i,g_j) $ for all pairs $ (i,j) $. 

A symmetric pairing $ \phi $ on $ G $ induces a map from $ G\to\Hom(G,\Q/\Z) $ taking $ g\mapsto\phi(g,\cdotp) $ (if the pairing were not symmetric we would have two maps---one for each factor). Such a pairing is said to be \emph{nondegenerate} if the induced map $ G\to\Hom(G,\Q/\Z) $ is injective. If the induced map is not injective, the pairing is said to be \emph{degenerate}, and the kernel of the map is called the \emph{kernel} of the pairing. The \emph{cokernel} of the pairing $ \phi $, denoted $ \coker(\phi) $, is defined to be $ G/\ker(\phi) $. Note that there is a natural nondegenerate symmetric pairing on $ \coker(\phi) $ induced by $ \phi $. 

If the map $ G\to\hG $ induced by $ \phi $ is an isomorphism, then the pairing is said to be \emph{perfect}. Recall that a perfect symmetric pairing on a group is called a \emph{duality pairing}. If $ G $ is finite, then a nondegenerate pairing on $ G $ is automatically perfect (this fact follows from counting). Given two groups $ G $ and $ H $ equipped with pairings $ \p_G $ and $ \p_H $, there is a notion of isomorphism between $ (G,\p_G) $ and $ (H,\p_H) $: an isomorphism between $ (G,\p_G) $ and $ (H,\p_H) $ is an isomorphism of groups $ f:G\to H $ such that $ \ip{f\inverse(h),f\inverse(h')}_G=\ip{h,h'}_H $ for all pairs $ (h,h')\in H\times H $. If $ (G,\p_G) $ is a group equipped with a pairing, an automorphism of $ G $ is said to \emph{respect the pairing} if it is an isomorphism of groups and pairings between $ (G,\p_G) $ and itself. Denote the set of automorphisms of $ G $ that respect the pairing $ \p_G $ by $ \Aut(G,\p_G) $.

\subsection{Duality}
Let $ R=\Z/b\Z $. When the ring $ R $ is understood, for an $ R $-module $ G $, we may consider its \emph{dual}, $ G^*=\Hom_R(G,R) $. If $ G $ is a finite abelian group, there is a canonical isomorphism between $ G^* $ and $ \hG $. For finite $ R $-modules $ G $ and $ H $, we have $ (G\oplus H)^*\simeq G^*\oplus H^* $. While $ G $ and $ G^* $ are not naturally isomorphic, when $ G $ is finite, there is a canonical isomorphism between $ G $ and $ (G^*)^* $ given by evaluation.

For a finite abelian group $ G $, the following are equivalent data: \begin{enumerate}
	\item a perfect symmetric pairing on $ G $;
	\item an isomorphism from $ G $ to $ \h{G} $;
	\item a perfect symmetric pairing on $ \hG $. 
\end{enumerate} Given a duality pairing $ \p_G $ on $ G $, let $ \phi $ denote the induced isomorphism from $ G\to\hG $. Then the duality pairing on $ \hG $ is given by $ \ip{x,y}=\ip{\phi\inverse(x),\phi\inverse(y)}_G $.

Let $ G $ and $ H $ be two finite abelian groups. Let $ f:G\to H $ be a group homomorphism. The \emph{transpose} of $ f $ is the map $ f^t:\h{H}\to \h{G} $ given by precomposing $ \ell\in\Hom(H,\Q/\Z) $ with $ f $ (a similar definition can be made if $ G $ and $ H $ are $ R $-modules for $ R=\Z/b\Z $ and $ f $ an $ R $-module homomorphism). If $ \hG $ is equipped with a pairing $ \phG $, then $ f^t $ pushes $ \phG $ forward to a pairing $ \p_{\h{H}} $ on $ \h{H} $ given by \[\ip{a,b}_{\h{H}}:=\ip{f^t(a),f^t(b)}_{\hG}.\] In other words, $ f^t $ induces a map from $ \Hom(\hG\otimes \hG,\Q/\Z)\to\Hom(\h{H}\otimes\h{H},\Q/\Z) $; let $ f^t(\p_{\h{G}}) $ denote the image of $ \p_{\h{G}} $ under this map (denoted $ \p_{\h{H}} $ above). Note that we do not require $ \phG $ to be perfect. 

Let $ (G,\p_{\h{G}}) $ and $ (H,\p_{\h{H}}) $ be groups whose duals are equipped with pairings. Define \[\Hom((G,\p_{\hG}),(H,\p_{\h{H}}))\] to be the set of group homomorphisms $ f:G\to H $ such that $ f^t(\p_{\hG})=\p_{\h{H}} $. In the category $ \calc $ defined in \Cref{Introduction}, $ \Hom((G,\p_{\hG}),(H,\p_{\h{H}})) $ is the set of morphisms between $ (G,\p_{\hG}) $ and $ (H,\p_{\h{H}}) $. When the pairings on both groups are understood, we denote $ \Hom((G,\p_{\h{G}}),(H,\p_{\h{H}})) $ by $ \Hom^*(G,H) $ for the sake of brevity. Note that this generalizes the notion of isomorphism between groups with duality pairings defined in the previous subsection. Moreover, given finite abelian groups $ G $ and $ H $ with duality pairings $ \p_{G} $ and $ \p_H $ and induced duality pairings $ \p_{\h{G}} $ and $ \p_{\h{H}} $ on the dual groups, note that $ (G,\p_G)\simeq(H,\p_H) $ if and only if $ (G,\p_{\h{G}})\simeq(H,\p_{\h{H}}) $.

\subsection{Cokernels of matrices}
Let $ M_n(R) $ denote the space of $ n\times n $ matrices with entries in a commutative ring $ R $. For $ M\in M_n(R) $, let $ \col(M) $ denote the column space of $ M $ (this is the image of the map $ R^n\to R^n $ given by $ M $). Define the \emph{cokernel} of $ M $ to be \[\cok(M)=R^n/\col(M).\] 

For $ M\in M_n(\Z) $, regard $ M $ as a linear map $ M:\Z^n\to \Z^n $. Then $ \cok(M) $ is a finitely generated abelian group; let $ \tcok(M) $ denote its torsion subgroup. For any subgroup $ Y\subset \Z^n $, let $ Y^\perp $ denote its orthogonal with respect to the standard scalar product on $ \Z^n $. 

If $ M $ is symmetric, $ \col(M)\subset\ker(M)^\perp $. Because $ \rank(\col(M))=\rank(\ker(M)^\perp) $ and because $ \ker(M)^\perp $ is a saturated submodule of $ \Z^n $ (i.e., if $ rx\in\ker(M)^\perp $ for $ r\in \Z $ and $ x\in \Z^n $, then $ x\in\ker(M)^\perp $), it follows that \[\tcok(M)=\ker(M)^\perp/\col(M).\] Since $ \tcok(M) $ is pure torsion and $ \cok(M) $ is finitely generated, $ \tcok(M) $ is finite.

\subsection{Sandpile groups} Recall the definition of the graph Laplacian. Let $ [n]=\{1,\ldots,n\} $, and let $ \Gamma $ be a graph on $ n $ vertices labeled by $ [n] $. The \emph{Laplacian} of $ \Gamma $, denoted $ L_\Gamma $, is the $ n\times n $ matrix whose $ (i,j) $ entry is given by \[(L_\Gamma)_{ij}=\begin{cases}
	1 & \mathrm{if\ }i\neq j\ \mathrm{and\ }\{i,j\}\ \mathrm{is\ an\ edge\ of\ }\Gamma;\\
	0 &\mathrm{if\ }i\neq j\ \mathrm{and\ }\{i,j\}\ \mathrm{is\ not\ an\ edge\ of\ }\Gamma;\\
	-\deg(i) & \mathrm{if\ }i=j.
\end{cases}\] Note that $ L_\Gamma $ is a symmetric matrix in $ M_n(\Z) $. If $ \Gamma $ is connected, then $ L_\Gamma $ has rank $ n-1 $ and kernel spanned by the all-ones vector (see \cite{Shokrieh}). If $ Z\subset\Z^n $ is the subgroup of elements whose coordinates sum to 0, then we see that $ \col(L_\Gamma)\subset Z $. Define the sandpile group of $ \Gamma $, denoted $ S_\Gamma $, to be $ Z/\col(L_\Gamma) $. We see immediately that $ S_\Gamma $ must be finitely generated. Moreover, it is finite if and only if $ \Gamma $ is connected.

\subsection{Random graphs} Let $ \Gamma\in G(n,q) $ denote that $ \Gamma $ is an Erd\H{o}s-R\'{e}nyi random graph on $ n $ labeled vertices with edges occurring independently with probability $ q $.
\subsection{Notation} We use $ |\cdot| $ or $ \# $ to denote the cardinality of sets (we use $ \# $'s to avoid confusion with the notation $ | $ for ``divides''; while absolute value signs take up less space). Throughout, $ \simeq $ denotes ``is isomorphic to,'' and $ \Prob $ and $ \E $ denote probability and expected value, respectively. Also, $ p $ will always denote a prime. For any integers $ i $ and $ j $, let $ \delta_{ij} $ be the Kronecker delta---if $ i=j $, then $ \delta_{ij}=1 $ and if $ i\neq j $, then $ \delta_{ij}=0 $. 

\section{Pairings associated to symmetric matrices}\label{sec: Pairings}
It is a well-known fact (see, e.g., \cite{Dauns}) that the cokernel of an invertible symmetric integral matrix $ M $ carries a canonical duality pairing. If $ x,y\in\cok(M) $, and $ X,Y\in\Z^n $ are lifts of $ x,y $ respectively, then the pairing is given by \[(x,y)\mapsto X^tM\inverse Y\mod\Z.\] In \cite{BL02}, Bosch and Lorenzini showed that the torsion subgroup of $ \cok(M) $ is canonically equipped with a duality pairing. When the matrix is invertible, this duality pairing coincides with the one defined above. 

We define this canonical duality pairing on $ \tcok(M) $ in the following way (we refer the reader to Section 1 of \cite{BL02} for a more general and detailed setup). Let $ \tau,\tau'\in\tcok(M) $, and let $ T $ and $ T' $ be lifts of $ \tau $ and $ \tau' $ to $ \ker(M)^\perp $, respectively. Since $ \tcok(M) $ is pure torsion, there exist $ k,k'\in\N $ such that $ kT,k'T'\in\col(M) $, i.e., there exist $ S,S'\in\Z^n $ such that $ MS=kT $ and $ MS'=k'T' $. Define a map $ \p_{\tcok(M)}:\tcok(M)\times\tcok(M)\to\Q/\Z $ by \[(\tau,\tau')\mapsto\frac{S^tMS'}{kk'}\mod\Z.\] If $ M $ is nonsingular, then $ \cok(M)=\tcok(M) $, and this map is the classical pairing on $ \cok(M) $ discussed above. The map $ \p_{\tcok(M)} $ is a duality pairing on $ \tcok(M) $: \begin{lemma}[Lemma 1.1 and Theorem 1.3 of \cite{BL02}]\label{pairing on torsion}
	The map $ \p_{\tcok(M)} $ is well-defined, bilinear, symmetric, and perfect. 
\end{lemma} 

When $ M $ is singular, $ \cok(M) $ has some free part to which $ \p_{\tcok(M)} $ does not naturally extend. The duality pairing on $ \tcok(M) $ implies the existence of a corresponding duality pairing on $ \htcok(M) $. Now, $ \htcok(M) $ is a quotient of $ \hcok(M) $, so the pairing on $ \htcok(M) $ does induce a natural pairing on $ \hcok(M) $. However, it can be quite difficult to compute this induced pairing on $ \hcok(M) $ using the definition given above. So, we will define a symmetric pairing on $ \hcok(M) $ given by $ M $; show that this pairing induces a perfect symmetric pairing on $ \htcok(M) $; and prove that the induced duality pairing on $ \tcok(M) $ coincides with the pairing of Bosch and Lorenzini. 

Recall $ M $ is a linear map $ \Z^n\to\Z^n $. Any $ \phi\in\hcok(M) $ can be extended to $ \phi:\Z^n\to\Q/\Z $ via the quotient $ \Z^n\to\cok(M) $. Note that $ \hcok(M)=\Hom(\cok(M),\Q/\Z) $ is the subset of maps $ \phi\in\Hom(\Z^n,\Q/\Z) $ such that $ \phi M=0 $ (i.e., for any lift $ \Phi $ of $ \phi $ to $ \Hom(\Z^n,\Q) $, we have $ (\Phi M)^t\in\Z^n $). For $ \tau\in\cok(M) $, let $ T $ be a lift of $ \tau $ to $ \Z^n $. Then for $ \phi\in\hcok(M) $, we have $ \phi(\tau)=\phi(T) $. We define a pairing on $ \hcok(M) $ given by the matrix $ M $ in the following way. For $ x,y\in\hcok(M) $, let $ X,Y\in\Hom(\Z^n,\Q) $ be lifts of $ x,y $, respectively. We define a map $ \p_M:\hcok(M)\times\hcok(M)\to\Q/\Z $ by \[(x,y)\mapsto XMY^t\mod\Z.\]

\begin{lemma}
	The map $ \p_M $ on $ \hcok(M) $ is well-defined, bilinear, and symmetric. 
\end{lemma}
\begin{proof}
	To show that $ \p_M $ is well-defined, we simply need to check that our map does not depend on the chosen lifts of $ x $ and $ y $. Recall that both $ MX^t $ and $ MY^t $ are vectors in $ \Z^n $. Hence, for $ A,B\in\Hom(\Z^n,\Z) $, we have \[(X+A)M(Y+B)^t=XMY^t+XMB^t+AMY^t+AMB^t=XMY^t\mod\Z.\] Hence, the map is well-defined. 
	
	Bilinearity and symmetry are both easily verified, where the symmetry of $ \p_M $ follows from the symmetry of $ M $ itself.  
\end{proof}

Since the Pontryagin duality functor $ \Hom(-,\Q/\Z) $ is exact, the inclusion $ \iota:\tcok(M)\hookrightarrow\cok(M) $ becomes a quotient $ \iota^t:\hcok(M)\twoheadrightarrow\htcok(M) $ after taking the transpose. The kernel of this quotient is the annihilator of $ \tcok(M) $, denoted $ \ann(\tcok(M)) $. It turns out that $ \p_M $ descends to a well-defined duality pairing on $ \htcok(M) $. Recall the definition of $ \coker\p_M $ from \Cref{subsec: pairings}.
\begin{lemma}\label{pairing on dual of cok is symmetric}
	Let $ M $ be a symmetric integral matrix and $ \p_M $ the symmetric pairing on $ \hcok(M) $. Then $ \ker(\p_M)=\ann(\tcok(M)) $, implying $ \coker(\p_M)=\htcok(M) $.
\end{lemma}
\begin{proof}
	Suppose $ \psi\in\ker(\p_M) $ so that for all $ x\in\hcok(M) $ we have $ \ip{\psi,x}_M=0 $. Let $ \tau\in\tcok(M) $ and choose a lift $ T $ of $ \tau $ to $ \ker(M)^\perp\subset\Z^n $. Since $ \tau\in\tcok(M) $, there exist $ k\in\N $ and $ S\in\Z^n $ such that $ kT=MS $. Consider the image of $ (S/k)^t\in\Hom(\Z^n,\Q) $ in $ \Hom(\Z^n,\Q/\Z) $, which we denote by $ [(S/k)^t] $. Since $ (S/k)^tM=T^t $, we may view $  [(S/k)^t] $ as an element of $ \hcok(M) $. Thus, if $ \Psi $ is any lift of $ \psi $ to $ \Q $, then \[\psi(\tau)=\psi(T)=\psi\left(\frac{MS}{k}\right)=\frac{\Psi MS}{k}=\ip{\psi,[(S/k)^t]}_M=0\mod\Z,\] forcing $ \psi\in\ann(\tcok(M)). $
	
	For the reverse inclusion, assume $ \phi\in\ann(\tcok(M)) $ so that for all $ \tau\in\tcok(M) $, we have $ \phi(\tau)=0 $. We would like to show that for all $ x\in\hcok(M) $ \[\ip{\phi,x}_M=\Phi MX^t=0\mod\Z,\] where $ \Phi $ and $ X $ are lifts of $ \phi $ and $ x $ to $ \Q $, respectively. Suppose $ x\in\hcok(M) $ so that $ MX^t\in\Z^n $. Writing $ X=Y/k $ for some $ k\in\N $ and $ Y^t\in\Z^n $, we see that $ kMX^t=MY^t\in\col(M) $. Hence, the image of $ MX^t $ in $ \cok(M) $ lies in $ \tcok(M) $. Since $ \phi\in\ann(\tcok(M)) $, \[\Phi MX^t=\phi([MX^t])=0\mod\Z,\] where $ [MX^t] $ denotes the image of $ MX^t $ in $ \cok(M) $. Therefore, $ \phi\in\ker(\p_M) $. This forces $ \ker(\p_M)=\ann(\tcok(M)) $, as desired. 
\end{proof}

By abuse of notation, we also use $ \p_M $ to denote the duality pairing on $ \htcok(M)=\coker\p_M $.
\begin{corollary}
	The pairing $ \p_{M} $ is a duality pairing on $ \htcok(M) $.
\end{corollary}

Since $ \htcok(M) $ is a finite abelian group, the duality pairing $ \p_{M} $ on $ \htcok(M) $ induces a duality pairing on $ \Hom(\htcok(M),\Q/\Z)=\tcok(M) $. The following shows that the duality pairing on $ \tcok(M) $ induced by $ \p_{M} $ is Bosch and Lorenzini's pairing $ \p_{\tcok(M)} $ on $ \tcok(M) $. 

\begin{lemma}\label{pairing same as Bosch-Lorenzini}
	The duality pairing on $ \tcok(M) $ induced by $ \p_{M} $ is $ \p_{\tcok(M)} $. 
\end{lemma}
\begin{proof}
	Recall that $ \p_{M} $ gives an isomorphism, $ \rho:\htcok(M)\to\tcok(M) $ taking $ \overline{\phi}\in\htcok(M) $ to $ \ip{\overline{\phi},-}_{M} $. Let $ \phi $ be any lift of $ \overline{\phi} $ to $ \hcok(M) $ (recall that there is a natural quotient $ \hcok(M)\to\htcok(M) $). If $ \Phi $ is any lift of $ \phi $ to $ \Hom(\Z^n,\Q) $, then $ \rho(\overline{\phi})=\ip{\overline{\phi},-}_{M}=[M\Phi^t] $, where the brackets denote the image of $ M\Phi^t\in\Z^n $ in $ \tcok(M) $. 
	
	For $ \tau,\tau'\in\tcok(M) $, let $ T,T' $ be lifts of $ \tau,\tau' $ to $ \ker(M)^\perp $. There exist $ k,k'\in\N $ and $ S,S'\in\Z^n $ such that $ kT=MS $ and $ k'T'=MS' $. Thus, the preimages of $ \tau $ and $ \tau' $ under $\rho$ are the images of $ S^t/k $ and $ (S')^t/k' $ in $ \htcok(M) $. The pairing on $ \tcok(M) $ induced by $ \p_{M} $ takes \[(\tau,\tau')\mapsto\ip{\rho\inverse(\tau),\rho\inverse(\tau')}_{M}=\frac{S^tMS'}{kk'}\mod\Z.\] However, recall that this is exactly $ \ip{\tau,\tau'}_{\tcok(M)} $.
\end{proof}

Later, given a finitely generated abelian group $ G $, we will be interested in taking the tensor product of $ G $ with $ \Z/b\Z $ for some positive integer $ b $. If $ \hG $ is equipped with a symmetric pairing $ \phG $, we would like the dual of $ G\otimes\Z/b\Z $ to come naturally equipped with a symmetric pairing given by $ \phG $. Let $ \mathfrak{R}=\Z/b\Z $. There is a natural quotient map $ G\to G\otimes\mathfrak{R} $ given by $ g\mapsto g\otimes 1 $, and therefore there is a natural inclusion $ \Hom(G\otimes\mathfrak{R},\Q/\Z)\to\hG $. Thus, we obtain a pairing on the dual of $ G\otimes \mathfrak{R} $ by restricting the pairing on $ \hG $ to $ \Hom(G\otimes \mathfrak{R},\Q/\Z)\subset\hG $. We abuse notation and also denote this pairing on $ \Hom(G\otimes \mathfrak{R},\Q/\Z) $ by $ \phG $. 

For a graph $ \Gamma $, recall from \Cref{Background} the definition of its graph Laplacian $ L_\Gamma $ and associated sandpile group $ S_\Gamma $. If $ \Gamma $ is connected, in the language of the above, we see that $ S_\Gamma $ is simply the torsion subgroup of the cokernel of the Laplacian, since $ \ker(L_\Gamma)^\perp=Z $ (this is a well-known fact \cite{Biggs93}; recall that $ Z $ is the subspace of vectors in $ \Z^n $ whose coordinates sum to 1). Hence, using Bosch and Lorenzini's construction, $ S_\Gamma $ comes equipped with a canonical duality pairing. 

However, since we will be dealing with Erd\H{o}s-R\'{e}nyi random graphs later on, we would like for there to exist a pairing on $ \h{S_\Gamma} $ for any graph $ \Gamma $. We claim that for any $ \Gamma $, not necessarily connected, $ \h{S_\Gamma} $ is equipped with a natural symmetric pairing given by $ L_\Gamma $, which we will denote using $ \p_{\h{S_\Gamma}} $. To see why, recall that we have the following chain of inclusions: \[\tcok(L_\Gamma)=\ker(L_\Gamma)^\perp/\col(L_\Gamma)\subset S_\Gamma=Z/\col(L_\Gamma)\subset\cok(L_\Gamma)=\Z^n/\col(L_\Gamma).\] As a result, the canonical duality pairing on $ \htcok(L_\Gamma) $ pushes forward to a symmetric pairing on $ \h{S_\Gamma} $, which itself pushes forward to the symmetric pairing $ \p_{L_\Gamma} $ on $ \hcok(L_\Gamma) $. Of these three pairings, we will be most interested in the symmetric pairing on $ \h{S_\Gamma} $. 

The following lemma will be useful later.
\begin{lemma}\label{check pairing on cokernel}
	Let $ \Gamma $ be a graph on $ n $ vertices and $ G $ any finite abelian group whose dual is equipped with a symmetric pairing $ \phG $. Suppose $ F:\Z^n\to G $ is a surjection whose restriction to $ Z $ remains onto. Also assume that $ FL_\Gamma=0 $ so that $ F $ descends to a map $ F:S_\Gamma=Z/\col(L_\Gamma)\to G $ with transpose $ F^t:\hG\to\h{S_\Gamma} $. Then $ F^t(\p_{\h{S_\Gamma}})=\phG $ if and only if $ F^t(\p_{L_\Gamma})=\phG $. 
\end{lemma}
\begin{proof}
	We start with the following commutative diagram and its dual: \begin{center}
		\begin{tikzcd}
			Z\arrow[d,two heads]\arrow[r,hookrightarrow] & \Z^n\arrow[d,two heads]\\
			S_\Gamma\arrow[r,hookrightarrow]\arrow[dr,two heads,"F" below left] & \cok(L_\Gamma)\arrow[d,two heads,"F" right]\\
			& G
		\end{tikzcd}
		$ \qquad $\begin{tikzcd}
			\h{Z} & (\Q/\Z)^n\arrow[l,two heads]\\
			\h{S_\Gamma}\arrow[u,hook] & \hcok(L_\Gamma)\arrow[u,hook]\arrow[l,two heads]\\
			& \hG\arrow[ul,hook,"F^t" below left]\arrow[u,hook,"F^t" right]
		\end{tikzcd}
	\end{center} Recall that $ F^t $ pushes the pairing $ \p_{\h{S_\Gamma}} $ on $ \h{S_\Gamma} $ forward to some pairing on $ \hG $. Likewise, the transpose of the inclusion $ \iota:S_\Gamma\hookrightarrow\cok(L_\Gamma) $ pushes the pairing $ \p_{\h{S_\Gamma}} $ forward to the pairing $ \p_{L_\Gamma} $ on $ \hcok(L_\Gamma) $. The commutativity of the diagram implies the result.
\end{proof}

\section{Obtaining the moments I: Finding the equations}\label{sec: Moments I}
Let $ G $ be a finite abelian group whose dual $ \hG=\Hom(G,\Q/\Z) $ is equipped with a symmetric pairing $ \phG $. Suppose $ \Gamma\in G(n,q) $, and let $ S $ denote its sandpile group $ S_\Gamma $ and $ L $ the Laplacian $ L_\Gamma $. Because $ S $ is defined as a quotient of $ Z $, any surjection $ S\to G $ lifts to a surjection $ Z\to G $. Therefore, \[\E(\#\Sur^*(S,G))=\sum_{F\in\Sur(Z,G)}\Prob(\col(L)\subset\ker(F)\mathrm{\ and\ }F^t(\p_{\h{S}})=\p_{\hG}).\] We will estimate the probabilities on the right-hand side of the above; to do so, we utilize the following more general setup. 

Let $ b $ denote the exponent of $ G $, and let $ a=b^2 $. Let $ R $ be the ring $ \Z/a\Z $. This notation will be used through \Cref{sec: Moments V}. Note that $ \Sur(S,G)=\Sur(S\otimes\Z/b\Z,G) $, i.e., whether $ \col(L)\subset\ker(F) $ only depends on the entries of $ L $ modulo $ b $. We will see later in this section that whether $ F^t(\p_{\h{S}})=\phG $ also only depends on the entries of $ L $ modulo $ b $. However, for reasons that will become apparent later, we will mostly work over $ R $ (modulo $ a $) in the following. In this and the following four sections, all of our linear algebra will be done over $ R $. However, since $ R $ is not a domain, we are forced to work more abstractly rather than just with matrices.

Let $ G $ be the finite abelian group fixed above with symmetric pairing $ \phG $ on its dual $ \hG $. Let $ G_p $ denote the Sylow $ p $-subgroup of $ G $, and write $ G=\bigoplus_p G_p $. For a prime $ p $, suppose $ G_p $ is a $ p $-group of type $ \lambda $, where $ \lambda $ is the partition given by $ \lambda_1\geq\lambda_2\geq\cdots\geq\lambda_r $. Because a pairing on $ \hG $ may be viewed as an element of $ \Hom(\hG\otimes\hG,\Q/\Z) $, any pairing on $ \hG $ is entirely determined by its restriction to $ (\hG)_p\otimes(\hG)_p $ for each $ p $. 

Let $ g_1,\ldots,g_r $ generate $ G_p $ as an abelian group with $ p^{\lambda_i}g_i=0 $ and no other relations, and let $ \h{g_i} $ be the elements of $ \hG_p $ such that $ \h{g_i}(g_j)=\delta_{ij}/p^{\lambda_i} $, where recall $ \delta_{ij} $ is the indicator of $ i=j $. Also, suppose that $ \ip{\h{g_i},\h{g_j}}_{\hG}=a_{ij}/p^{\lambda_j} $ for $ 1\leq i\leq j\leq r $, where $ a_{ij} $ is an integer $ 0\leq a_{ij}< p^{\lambda_j} $. Note that the restriction of the pairing to $ (\hG)_p\otimes(\hG)_p $ is entirely determined by these values $ a_{ij} $. Let $ \p_{\hG_p} $ denote the restriction of the pairing $ \phG $ to $ (\hG)_p\otimes(\hG)_p $.

Recall $ R=\Z/a\Z $, and let $ V=R^n $. Suppose $ M $ is a symmetric matrix in $ M_n(R) $ with entries $ m_{ij}\in R $ for $ 1\leq i\leq j\leq n $. We distinguish a basis of $ V $, say $ v_1,\ldots,v_n $; let $ \h{v_1},\ldots,\h{v_n} $ denote the corresponding dual basis of $ \h{V} $, where $ \h{v_i}(v_j)=\delta_{ij}/a $. With respect to the bases $ \h{v_1},\ldots,\h{v_n} $ and $ v_1,\ldots,v_n $, we may view $ M $ as a linear map $ \h{V}\to V $ where $ M(\h{v_i})=\sum_{j}m_{ij}v_j $. Recall from \Cref{sec: Pairings} that $ \hcok(M)=\Hom(\cok(M),\Q/\Z) $ comes equipped with a symmetric pairing given by $ M $, which we'll denote by $ \p_M $. Let $ F:V\to G_p $ be a surjective homomorphism and denote the transpose of $ F $ by $ F^t:\hG_p\to\h{V} $.  Since $ \cok(M) $ is a quotient of $ R^n $, if $ \col(M)\subset\ker(F) $, then $ F $ factors through the quotient of $ V $ by $ \col(M) $, i.e., $ F $ descends to a map $ F:\cok(M)\to G_p $. Let $ F(v_j)=\sum_{i=1}^rf_{ij}g_i $ for $ f_{ij}\in\Z/p^{\lambda_i}\Z $. 

\begin{lemma}\label{lemma: cok eqns}
	With notation as above, $ \col(M)\subset\ker(F) $ if and only if \begin{equation}\label{cok eqns}
		\sum_{k=1}^nf_{ik}m_{k\ell}=0\mod p^{\lambda_i}
	\end{equation} for $ 1\leq i\leq r $ and $ 1\leq\ell\leq n $. If $ \col(M)\subset\ker(F) $, then $ F $ descends to a map $ \cok(M)\to G_p $, and we have $ F^t(\p_M)=\p_{\hG_p} $ if and only if  \begin{equation}\label{pairing eqns}
		\sum_{k,\ell=1}^nf_{ik}m_{k\ell}f_{j\ell}=p^{\lambda_i}a_{ij}\mod p^{\lambda_i+\lambda_j}
	\end{equation} for all $ 1\leq i\leq j\leq r $ (we will use $ \eqref{cok eqns} $ to show that the left-hand side of \eqref{pairing eqns} is well-defined modulo $ p^{\lambda_i+\lambda_j} $). 
\end{lemma}
\begin{proof}
	We first note that $ \col(M)\subset\ker(F) $ if and only if the composition $ FM\in\Hom(\h{V},G_p) $ is 0. Moreover, we see that $ FM=0 $ if and only if \[\sum_{k=1}^nf_{ik}m_{k\ell}=0\mod p^{\lambda_i}\] for all $ 1\leq i\leq r $ and $ 1\leq \ell\leq n $. 
	
	Now, assume $ FM=0 $ so that $ F:V\to G_p $ descends to a map $ F:\cok(M)\to G_p $. Recall from \Cref{sec: Pairings} that $ F^t $ takes the pairing on $ \hcok(M) $ to a pairing on $ \h{G}_p $ and that a symmetric pairing on $ \hG_p $ is entirely determined by where it sends the pairs $ (\h{g_i},\h{g_j}) $ for $ 1\leq i\leq j\leq r $. Thus, we first compute $ F^t(\h{g_i})=\h{g_i}\circ F\in\h{V} $ for $ 1\leq i\leq r $. We see that \[\h{g_i}(F(v_j))=\h{g_i}\left(\sum_kf_{kj}g_k\right)=\frac{f_{ij}}{p^{\lambda_i}}\mod\Z.\] Hence, with respect to the aforementioned bases, we may view $ F^t(\h{g_i}) $ as the row vector \[\bpm{f_{i1}/p^{\lambda_i}&\cdots&f_{in}/p^{\lambda_i}}\mod\Z.\] It follows that \[\ip{F^t(\h{g_i}),F^t(\h{g_j})}_M=F^t(\h{g_i})MF^t(\h{g_j})^t=\sum_{k,\ell=1}^n\frac{f_{ik}m_{k\ell}f_{j\ell}}{p^{\lambda_i+\lambda_j}}\mod\Z\] (recall we can do this computation by lifting $ F^t(\h{g_i}) $ to $ \Q $). Therefore, given that $ FM=0 $, we see that $ F^t(\p_M)=\p_{\hG_p} $ if and only if \[\sum_{k,\ell=1}^n\frac{f_{ik}m_{k\ell}f_{j\ell}}{p^{\lambda_i+\lambda_j}}=\frac{a_{ij}}{p^{\lambda_j}}\mod\Z\] for all $ 1\leq i\leq j\leq r $.
\end{proof}

Note that $ \Sur(\Z^n,G)=\Sur(\Z^n\otimes R,G)=\Sur(V,G) $. Hence, given $ F\in\Sur(V,G) $, \Cref{lemma: cok eqns} and the fact that $ p^{2\lambda_1}|a $ imply whether $ \col(M)\subset\ker(F) $ and $ F^t(\p_M)=\phG $ depends only on the entries of $ M $ modulo $ a $.  As before, let $ V $ denote the $ R $-module $ R^n $ with distinguished basis $ v_1,\ldots,v_n $. Consider $ V^*=\Hom(V,R) $ with dual basis $ v_1^*,\ldots,v_n^* $, where $ v_i^*(v_j)=\delta_{ij} $. Note that there is an isomorphism between $ V^* $ and $ \h{V} $ given by sending $ v_i^* $ to $ \h{v_i} $. 

Recall $ G_p $ is a $ p $-group of type $ \lambda $. Define $ H_p=(\Z/p^{2\lambda_1}\Z)^r $, and, noting that $ \lambda_1 $ and $ r $ depend implicitly on $ p $, set $ H=\bigoplus_p H_p $. For any $ p|\# G $, let $ h_1,\ldots,h_r $ denote the standard basis of $ H_p=R_p^r $, and define a surjective map $ H\to G $ by summing the maps $ H_p\to G_p $ taking $ h_i\mapsto g_i $ for each $ p $. We will use this map throughout the argument---unless stated otherwise, when we refer to a surjection $ H\to G $, we refer to this map. 

Let $ F\in\Hom(V,G) $, and let $ \calf\in\Hom(V,H) $ be a lift of $ F $ to $ H $ (with respect to the aforementioned surjection $ H\to G $). 
\begin{remark}
	For the reader disturbed by this arbitrary choice of lift, we remark here that while there are many choices for $ \calf $, our arguments in the following section do not depend on this choice---the proofs only depend on the values of $ \calf $ in $ H $ after projecting to the quotient $ G $, i.e., the proofs only depend on $ F $ itself. The lift $ \calf $ solves the issue of whether the equations \eqref{pairing eqns} from \Cref{lemma: cok eqns} are well-defined. By considering a lift of $ \calf $, we can instead work with equations that are always well-defined. 
	
	For example, consider the $ 1\times 1 $ matrix $ M=[ep]\in M_1(\Z/p^2\Z) $ for some $ e\neq0 $ mod $ p $. The cokernel of $ M $ can be written as \[\hcok(M)=\{\phi\in\Hom(\Z/p^2\Z,\Q/\Z)\;|\;\phi M=0\}=\{\phi\in\Hom(\Z/p^2\Z,\Q/\Z)\;|\;p\phi=0\}\] and is isomorphic to $ \Z/p\Z $. The pairing on $ \hcok(M) $ is given by \[\langle x/p,y/p\rangle_M=\frac{xepy}{p^2}=\frac{xey}{p}\in\Q/\Z,\] and we see that $ \p_M $ depends on the entries of $ M $ modulo $ p^2 $ rather than the entries of $ M $ modulo $ p $.  
\end{remark}

Consider the Sylow $ p $-part of $ \im(\calf) $, i.e., $ \calf:V\to H_p $. Let $ F(v_j)=\sum_{i=1}^rf_{ij}g_i\in G_p $, where $ f_{ij}\in\Z/p^{\lambda_i}\Z $, and let $ \calf(v_j)=\sum_{i=1}^rF_{ij}h_i\in H_p$, where $ F_{ij}\in R=\Z/p^{2\lambda_1}\Z $. Since $ \calf $ is a lift of $ F $, note that $ F_{ij}\equiv f_{ij} $ modulo $ p^{\lambda_i} $.

Define $ \Lambda_p:H_p\to H_p $ to be the map sending $ h_i\mapsto p^{2\lambda_1-\lambda_i}h_i $ for $ 1\leq i\leq r $. Let $ \Lambda:H\to H $ be the direct sum of the $ \Lambda_p $'s. Note that the image of $ \Lambda $ is isomorphic to $ G $, giving us a copy of $ G $ living in $ H $---i.e., the map $ \im\Lambda\to G $ given by $ p^{2\lambda_1-\lambda_i}h_i\mapsto g_i $ is an isomorphism. Moreover, the surjection $ H\to G $ defined earlier is equal to the composition of $ \Lambda $ with this isomorphism $ \im\Lambda\simeq G $. Define $ \Omega_p:H_p\to H_p $ to be the map taking $ h_i\mapsto p^{\lambda_1-\lambda_i}h_i $ for $ 1\leq i\leq r $. Let $ \Omega:H\to H $ be the sum of the $ \Omega_p $'s. 

We define an element $ A $ of $ \Sym_2 H\subset\Hom(H^*,H) $ that is equivalent to the data of the pairing $ \phG $. For each $ p|\#G $, recall that $ \ip{\h{g_i},\h{g_j}}_{\hG}=a_{ij}/p^{\lambda_j} $ for all $ 1\leq i\leq j\leq r $. Let $ A_p $ denote the following symmetric element of $ \Hom(H_p^*,H_p) $, which we view as $ H_p\otimes H_p $. Let \[A_p=\sum_{1\leq i<j\leq r}p^{2\lambda_1-\lambda_j}a_{ij}(h_i\otimes h_j+h_j\otimes h_i)+\sum_{i=1}^rp^{2\lambda_1-\lambda_i}a_{ii}h_i\otimes h_i\in H_p\otimes H_p,\] and let $ A=\bigoplus_p A_p $. We may think of $ A $ as a way of expressing the pairing $ \phG $ algebraically.      

Now, note that for each $ v_\ell^*\in V^* $, where $ 1\leq \ell\leq n $, we have \[\Lambda\calf M(v_\ell^*)=\sum_{p|\# G}\sum_{i=1}^r \left(p^{2\lambda_1-\lambda_i}\sum_{k=1}^nF_{ik}m_{k\ell}\right)h_i.\]  Note here that in the above we are abusing notation slightly---both the $ h_i $'s and $ r $ implicitly depend on $ p $. We note that 
\begin{equation}\label{cok eqns lifted}
	p^{2\lambda_1-\lambda_i}\sum_{k=1}^nF_{ik}m_{k\ell}=0\mod p^{2\lambda_1}
\end{equation} for $ 1\leq i\leq r $ and $ 1\leq\ell\leq n $ if and only if \eqref{cok eqns} holds. If we view $ G $ as the subgroup $ \Lambda H $ of $ H $, then note that $ \Lambda\calf $ is simply the map $ F:V\to G $ defined earlier. Likewise, we see that 
\[\Omega\calf M\calf^t\Omega^t(h_i^*)=\sum_{p|\# G}\sum_{1\leq j\leq r}\sum_{k,\ell=1}^np^{2\lambda_1-\lambda_i-\lambda_j}F_{ik}F_{j\ell}m_{k\ell}h_j.\] Note that $ \Omega\calf M\calf^t\Omega^t $ is a map $ H^*\to H $ and is symmetric when viewed as an element of $ H\otimes H $, since $ (\Omega\calf M\calf^t\Omega^t)^t=\Omega\calf M\calf^t\Omega^t $. Moreover, 
\begin{equation}\label{pairing eqns lifted}
	\sum_{k,\ell=1}^np^{2\lambda_1-\lambda_i-\lambda_j}F_{ik}F_{j\ell}m_{k\ell}=p^{2\lambda_1-\lambda_j}a_{ij}\mod p^{2\lambda_1}
\end{equation}  for $ 1\leq i\leq j\leq r $ if and only if \eqref{pairing eqns} holds. Hence, $ \Lambda\calf M=0 $ and $ \Omega\calf M\calf^t\Omega^t=A $ if and only if both \eqref{cok eqns} and \eqref{pairing eqns} hold for all $ p|\# G $. It follows that \begin{equation}\label{equivalent probabilities}
	\Prob(\col(M)\subset\ker(F)\ \mathrm{and\ }F^t(\p_{M})=\phG)=\Prob(\Lambda\calf M=0\mathrm{\ and\ }\Omega\calf M\calf^t\Omega^t=A).
\end{equation}

While working with the lift $ \calf $ seems artificial and notationally clumsy, we do so because it removes a dependence between our equations---we saw previously that the equations \eqref{pairing eqns} for $ F^t(\p_M)=\phG $ are not \emph{a priori} well-defined. By considering a lift of $ F $, the equations \eqref{pairing eqns lifted} can be studied without first having to condition on \eqref{cok eqns lifted}. The maps $ \Lambda $ and $ \Omega $ are introduced so that the lifted equations \eqref{cok eqns lifted} and \eqref{pairing eqns lifted} hold if and only if our original equations \eqref{cok eqns} and \eqref{pairing eqns} hold. Thus, we may consider all of the equations simultaneously; doing so will be helpful when we apply the discrete Fourier transform in the next section.

\section{Obtaining the moments II: The structural properties of the equations}\label{sec: Moments II}
In the previous section, we showed that for any random symmetric matrix $ M\in M_n(R) $, we have \[\E(\#\Sur^*(\cok(M),G))=\sum_{F\in\Sur(V,G)}\Prob(\Lambda\calf M=0\mathrm{\ and\ }\Omega\calf M\calf^t\Omega^t=A),\] where $ \calf $ is any lift of $ F $ to $ H $. The goal of this section and the two that follow is to estimate the probabilities in the right-hand side of the equation above. Let $ \zeta $ be a primitive $ a $th root of unity, and let $ \cale $ denote the event \[\Lambda\calf M=0\mathrm{\ and\ }\Omega\calf M\calf^t\Omega^t=A.\] Recall that we may view $ M $ as an element of $ \Hom(V^*,V) $ and note that $ \Lambda\calf M\in\Hom(V^*,G) $ (where we view $ G\simeq\Lambda H $ as a subgroup of $ H $) and that $ \Omega\calf M\calf^t\Omega^t\in\Hom(H^*,H) $, which, recall, is naturally isomorphic to $ H\otimes H $. Since $ (\Omega\calf M\calf^t\Omega^t)^t=\Omega\calf M\calf^t\Omega^t $, it follows that $ \Omega\calf M\calf^t\Omega^t\in\Sym_2 H $. Thus, the Fourier expansion implies \[\mathbbm{1}_{\cale}=\frac{1}{|G|^n\cdot|\Hom(\Sym_2 H,R)|}\sum_{\substack{C\in\Hom(\Hom(V^*,G),R)\\ D\in\Hom(\Sym_2 H,R)}}\zeta^{C(\Lambda\calf M)+D(\Omega\calf M\calf^t\Omega^t-A)}.\] Hence, \[\Prob(\cale)=\E(\mathbbm{1}_\cale)=\frac{1}{|G|^n\cdot|\Hom(\Sym_2 H,R)|}\sum_{\substack{C\in\Hom(\Hom(V^*,G),R)\\ D\in\Hom(\Sym_2 H,R)}}\E(\zeta^{C(\Lambda\calf M)+D(\Omega\calf M\calf^t\Omega^t-A)}).\] The pairs $ (C,D) $ give the equations that a matrix must satisfy in order for $ F $ to factor through the cokernel of the matrix and for $ F $ to push the pairing $ \p_M $ on $ \hcok(M) $ forward to the fixed pairing on $ \hG $. Viewing $ C(\Lambda\calf M)+D(\Omega\calf M\calf^t\Omega^t-A) $ as a function of $ M $, we see that it is in fact a linear function of the entries $ m_{ij} $ of $ M $ (for $ i\leq j $). In the following, we will compute the coefficient of each $ m_{ij} $. When the entries $ m_{ij} $ are independent, the expected value function is multiplicative, and we can factor the expected value in the right-hand side of the above (see \eqref{factored expected value} below). 

For the rest of this section, unless specified otherwise, when we refer to $ G $ we are referencing the copy of $ G $ living in $ H $ as $ \Lambda H $. Because we are viewing $ G $ as a subgroup of $ H $, we may view elements of $ H^* $ as elements of $ G^* $, since a linear map $ H\to R $ restricts to a linear map $ G=\Lambda H\to R $. Further, any map $ G\to R $ yields a map $ H\to R $ via precomposition with $ \Lambda $. However, note that an element of $ G^* $ cannot a priori be applied to an arbitrary element of $ H $.

Because $ V^*\simeq R^n $ (noncanonically), there is a natural isomorphism $ \Hom(V^*,R)\otimes G\to\Hom(V^*,G) $. Therefore, there is a natural isomorphism from $ \Hom(\Hom(V^*,G),R)\to\Hom(V\otimes G,R) $, which is itself canonically isomorphic to $ \Hom(V,\Hom(G,R))\simeq\Hom(V,G^*) $. Hence, we regard $ C $ as an element of $ \Hom(V,G^*) $. Let $ e:G^*\times G\to R $ denote the evaluation map. Similarly, there is a natural isomorphism between $ \Hom(\Sym_2 H,R) $ and $ \Sym^2 H^*$. Let $ [a^*\otimes b^*] $ denote the class of $ a^*\otimes b^*\in H^*\otimes H^* $ in $ \Sym^2 H^*$, and recall that $ \Sym_2 H $ is generated by the elements of the form $ a\otimes b+b\otimes a $ for $ a\neq b $ along with the elements $ a\otimes a $. The aforementioned isomorphism between $ \Hom(\Sym_2 H,R) $ and $ \Sym^2 H^* $ is witnessed by letting $ [a^*\otimes b^*] $ correspond to the map taking \[x\otimes y+y\otimes x\mapsto a^*(x)b^*(y)+a^*(y)b^*(x)\in R\quad\textrm{and}\quad x\otimes x\mapsto a^*(x)b^*(x)\in R.\]  Use this isomorphism to regard $ D $ as an element of $ \Sym^2 H^*$, and note that the evaluation does not depend on the chosen representative of $ D $ in $ H^*\otimes H^* $. Let $ \ev:\Sym^2 H^*\times\Sym_2 H\to R $ denote the evaluation map described above. For each element $ D $ of $ \Sym^2 H^*$ and choice of basis for $ H $, there exists some unique upper-triangular representative of $ D $ in $ H^*\otimes H^* $ with respect to this basis. We will use this fact to simplify computations.

The random matrices we will consider have entries that are independent \emph{with respect to a specific basis} of $ V $. Therefore, occasionally we are forced to do some computations using this distinguished basis, which we denote by $ v_1,\ldots,v_n $. At times, we will work with an alternate basis more conveniently aligned with $ G $ or $ H $ (through $ \calf $). Because we are working over $ R=\Z/a\Z $, which is not necessarily an integral domain, we prioritize making our arguments basis-independent as often as possible. 

We begin by computing the coefficient on each $ m_{ij} $ and rewriting it in more convenient notation. For $ h_1,h_2\in H $, let $ {h_1\odot h_2} $ be shorthand for the element $ h_1\otimes h_2+h_2\otimes h_1\in\Sym_2 H $. We see that \begin{align*}
	C(\Lambda\calf M)&+D(\Omega\calf M\calf^t\Omega^t)
	\\
	&=\sum_{i=1}^n\sum_{j=i+1}^n\Big(e(C(v_j),\Lambda\calf(v_i))+e(C(v_i),\Lambda\calf(v_j))+\ev(D,{\Omega\calf(v_i)\odot\Omega\calf(v_j)})\Big)m_{ij}\\
	&\qquad\qquad\qquad+\sum_{i=1}^n\Big(e(C(v_i),\Lambda\calf(v_i))+\ev(D,\Omega\calf(v_i)\otimes\Omega\calf(v_i))\Big)m_{ii}.
\end{align*} For $ i<j $, define \[E_{ij}(C,D,\calf)=e(C(v_j),\Lambda\calf(v_i))+e(C(v_i),\Lambda\calf(v_j))+\ev(D,{\Omega\calf(v_i)\odot\Omega\calf(v_j)}),\] and we also define \[E_{ii}(C,D,\calf)=e(C(v_i),\Lambda\calf(v_i))+\ev(D,\Omega\calf(v_i)\otimes\Omega\calf(v_i)).\] Hence, when the entries $ m_{ij} $ (for $ i\leq j $) are independent, \begin{equation}\label{factored expected value}
	\Prob(\Lambda\calf M=0\mathrm{\ and\ }\Omega\calf M\calf^t\Omega^t=A)=\frac{1}{|G|^n\cdot|\Sym^2 H|}\sum_{(C,D)}\E(\zeta^{-D(A)})\prod_{i\leq j}\E(\zeta^{E_{ij}(C,D,\calf)m_{ij}}).
\end{equation} 

For $ x\neq0 $, we know from \cite{Wood} that $ |\E(\zeta^{um_{ij}})| $ will be quite small (Lemma 4.2 of \cite{Wood}; restated as \Cref{expected value bound}). To get the bounds required to prove \Cref{sandpile moments intro}, we would like most of the coefficients $ E_{ij}(C,D,\calf) $ to be nonzero. Thus, this section is devoted to pinpointing the characteristics of $ \calf $ and $ (C,D) $ that control the number of nonzero $ E_{ij}(C,D,\calf) $. Note there are $ \binom{n+1}{2} $ total coefficients, and, to get the bounds we want, we would like quadratically many $ E_{ij}(C,D,\calf) $'s to be nonzero. For certain ``good'' $ F $, we will show that there are three possible outcomes: on the order of $ n^2 $ coefficients are nonzero; on the order of $ n $ coefficients are nonzero; all of the coefficients are zero. To show this, we identify the desired structural property of $ (C,D) $ that affects the number of nonzero coefficients. We also show that there are relatively few $ (C,D) $ such that only linearly many of the coefficients are nonzero; we count explicitly the $ (C,D) $ for which all of the coefficients are zero. In the following, we will assume $ F $ is ``good'' and define the aforementioned structural properties of $ (C,D) $. We will also explain the structural property which makes $ F $ ``good.'' Later, in \Cref{sec: Moments IV}, we will consider the remaining $ F $.

Now, we will write the coefficients $ E_{ij}(C,D,\calf) $ more equivariantly via a pairing. For elements $ \Lambda\calf\in\Hom(V,G) $ and $ C\in\Hom(V,G^*) $, define a map $ \phi_{\Lambda\calf,C}\in\Hom(V,G\oplus G^*) $ given by taking the direct sum of $ \Lambda\calf $ and $ C $. Likewise, define a map $ \phi_{C,\Lambda\calf}\in\Hom(V,G^*\oplus G) $ by switching the order of summation. There is a map \[t:(G\oplus G^*)\times(G^*\oplus G)\to R\] given by \[((h,\phi),(\phi',h'))\mapsto\phi'(h)+\phi(h').\] Note that for all $ u,v\in V $, \[t(\phi_{C,\Lambda\calf}(u),\phi_{\Lambda\calf,C}(v))=e(C(u),\Lambda\calf(v))+e(C(v),\Lambda\calf(u)).\] 

Similarly, for $ \calf\in\Hom(V,H) $ and $ C\in\Hom(V,G^*) $, we have a map $ \phi_{\calf,C}\in\Hom(V,H\oplus G^*) $ given by taking the direct sum of $ \calf $ and $ C $. Immediately, we see that for any submodule $ U $ of $ V $, \[\ker(\phi_{\calf,C}|_U)\subset\ker(\calf|_U).\] Given any subset $ \sigma\subset[n] $, let $ V_{\sigma} $ denote the submodule of $ V $ generated by the $ v_i $ such that $ i\not\in\sigma $. In other words, $ V_\sigma $ is the submodule given by not using the coordinates in $ \sigma $. The following definition gives the key structural property of the pair $ (C,D) $ (for some fixed $ \calf $) that determines if enough of the coefficients $ E_{ij}(C,D,\calf) $ are nonzero.
\begin{definition}
	Let $ 0<\gamma<1 $ be a real number (which we will specify later). For a fixed $ \calf\in\Hom(V,H) $, we say that the pair $ (C,D) $ is \emph{robust} for $ \calf  $ if every $ \sigma\subset[n] $ such that $ |\sigma|<\gamma n $ has the property that \[\ker(\phi_{\calf,C}|_{V_\sigma})\subsetneq\ker(\calf|_{V_\sigma}).\] Otherwise, the pair $ (C,D) $ is said to be \emph{weak} for $ \calf $. 
\end{definition}
We give an upper bound for the number of weak pairs $ (C,D) $. 
\begin{lemma}[Estimate for number of weak $ (C,D) $]\label{weak estimate}
	Given $ G $, there is a constant $ C_G $ such that for all $ n $ the following holds. Let $ H $ be defined as in \Cref{sec: Moments I}, and let $ \calf\in\Hom(V,H) $. The number of pairs $ (C,D)\in\Hom(V,G^*)\times\Sym^2 H^*$ that are weak for $ \calf $ is at most \[C_G{{n}\choose{\lceil\gamma n\rceil-1}}|G|^{\gamma n}.\]
\end{lemma}
\begin{proof}
	Suppose $ (C,D) $ is weak for $ \calf $. Then there exists some $ \sigma\subset [n] $ with $ |\sigma|<\gamma n $ such that \[\ker(\phi_{\calf,C}|_{V_\sigma})=\ker(\calf|_{V_\sigma}).\] Because $ \sigma\subset\mu $ implies $ V_{\mu}\subset V_{\sigma} $, we see that the above equality still holds for $ \sigma $ enlarged. Thus, without loss of generality, we may assume that $ |\sigma|=\lceil\gamma n\rceil-1 $. Now, since $ \ker(\phi_{\calf,C}|_{V_\sigma})=\ker(\calf|_{V_\sigma}) $, it follows that $ Cs $ is determined by $ \calf s $ for $ s\in V_{\sigma} $: if $ \calf s=\calf s' $ for $ s,s'\in V_\sigma $ and $ Cs\neq C{s'} $, then $ s-s'\in\ker(\calf) $ while $ s-s'\not\in\ker(C) $. In other words, for each value in $ \im(\calf|_{V_\sigma}) $ there is exactly one corresponding value in $ \im(C|_{V_\sigma}) $. Moreover, we have a well-defined homomorphism $ \psi:\im(\calf|_{V_\sigma})\to G^* $ taking $ \psi(\calf s)=Cs $ for all $ s\in V_{\sigma} $. 
	
	Recall that $ C $ is determined by the values $ Cv_i $ for $ i\in\sigma $ in addition to $ \im(C|_{V_\sigma}) $. Therefore, given $ \calf $, to specify a $ C\in\Hom(V,G^*) $ such that $ (C,D) $ is weak for $ \calf $ for any $ D\in\Sym^2 H^*$, it suffices to choose a subset $ \sigma\subset[n] $ with $ |\sigma|=\lceil\gamma n\rceil-1 $, values $ Cv_i $ for $ i\in\sigma $, and a homomorphism $ \psi:\im(\calf|_{V_\sigma})\to G^* $. Thus, there are $ {{n}\choose{\lceil\gamma n\rceil-1}} $ choices for $ \sigma $ and $ |G| $ choices for each $ Cv_i $ and $ |\Hom(\im(\calf|_{V_\sigma}),G^*)| $ choices for $ \psi $, and then $ C $ is determined. There are $ |\Sym^2 H| $ possibilities for $ D $. Note that because $ H $ is determined by $ G $ and because $ \im(\calf|_{V_\sigma})\subset H $, we can find some constant $ C_G $ such that $ |\Sym^2 H|\cdot|\Hom(\im(\calf|_{V_\sigma}),G^*)|\leq C_G $. 
\end{proof}

The following gives a sufficient condition for $ (C,D) $ to be weak in terms of $ \ev $ and $ t $.
\begin{lemma}\label{weak condition}
	Let $ \calf\in\Hom(V,H) $, let $ C\in\Hom(V,G^*) $, and let $ D\in\Sym^2H^* $. Suppose $ U $ is a submodule of $ V $ such that $ \calf U=H $. Then if $ U' $ is a submodule of $ V $ such that for all $ u\in U $ and $ u'\in U' $ \[t(\phi_{C,\Lambda\calf}(u),\phi_{\Lambda\calf,C}(u'))+\ev(D,\Omega\calf(u)\odot\Omega\calf(u'))=0,\] then $ \ker(\phi_{\calf,C}|_{U'})=\ker(\calf|_{U'}) $.
\end{lemma}
\begin{proof}
	Suppose for a contradiction that there exists some $ w\in U' $ with $ \calf(w)=0\in H $ but $ C(w)=\psi\neq0\in G^* $. Since $ \psi\neq0 $, there exists some element $ g\in G\subset H $ such that $ \psi(g)\neq0 $. Because $ \calf U=H $, we have $ \Lambda\calf U=\Lambda H=G $, and it follows that there must be some $ x\in U $ such that $ \Lambda\calf(x)=g $. Let $ C(x)=\psi' $. Then \begin{align*}
		t(\phi_{C,\Lambda\calf}(x),\phi_{\Lambda\calf,C}(w))+\ev(D,\Omega\calf(x)\odot\Omega\calf(w))&=\psi(g)+\psi'(0)+\ev(D,{\Omega\calf(x)\odot\Omega\calf(w)})\\
		&=\psi(g)+\ev(D,0)=\psi(g)\neq0.
	\end{align*} By hypothesis, \[t(\phi_{C,\Lambda\calf}(u),\phi_{\Lambda\calf,C}(u'))+\ev(D,\Omega\calf(u)\odot\Omega\calf(u'))=0,\] for any $ u\in U $ and $ u'\in U' $, so we have a contradiction. 
\end{proof}
\begin{corollary}\label{weak cor}
	Let $ \calf\in\Hom(V,H) $, let $ C\in\Hom(V,G^*) $, and let $ D\in\Sym^2H^* $. Suppose $ U $ is some submodule of $ V $ such that $ \calf U=H $. Let \[\sigma=\{i\in[n]\;|\;t(\phi_{C,\Lambda\calf}(u),\phi_{\Lambda\calf,C}(v_i))+\ev(D,\Omega\calf(u)\odot\Omega\calf(v_i))\neq0\ \textrm{for some }u\in U\}.\] Then $ \ker(\phi_{\calf,C}|_{V_\sigma})=\ker(\calf|_{V_\sigma}) $. In particular, if $ |\sigma|<\gamma n $, then $ (C,D) $ is weak for $ \calf $. 
\end{corollary}
\begin{proof}
	By assumption, $t(\phi_{C,\Lambda\calf}(u),\phi_{\Lambda\calf,C}(v))+\ev(D,\Omega\calf(u)\odot\Omega\calf(v))=0$ for all $ u\in U $ and $ v\in V_\sigma $. Applying \Cref{weak condition} with $ U'=V_\sigma $ tells us that $ \ker(\phi_{\calf,C}|_{V_\sigma})=\ker(\calf|_{V_\sigma}) $. 
\end{proof}

Now, we define the important structural property of $ F $, which is a generalization of the notion of a linear code to $ R $-modules. 
\begin{definition}
	For $ F\in\Hom(V,G) $, we call $ F $ a \emph{code of distance} $ w $ if for every $ \sigma\subset[n] $ such that $ |\sigma|<w $, we have $ FV_{\sigma}=G $. Said another way, $ F $ is not only a surjection, but would remain so if we threw out any fewer than $ w $ standard basis vectors from $ V $.\footnote{We note that this definition is dual to the usual notion of a code. If $ R $ is a field, then $ F $ is a code if and only if the transpose $ F^t:V^*\to G^* $ is injective and its image $ \im(F^t)\subset V^* $ is a code of distance $ w $ in the usual sense.}
\end{definition}

The lift $ \calf $ of a code $ F $ to $ H $ is still a code; importantly, note that the notion of being a code is a property we define for $ F\in\Hom(V,G) $ rather than for the lift of $ F $. However, the methods we use to bound the multiplicands in \eqref{factored expected value} require $ \calf $ to be a code as well.
\begin{lemma}\label{lift of a code is a code}
	Let $ F\in\Hom(V,G) $ be a code of distance $ w $, and let $ \calf\in\Hom(V,H) $ be a lift of $ F $ to $ H $. Then $ \calf $ is a code of the same distance.
\end{lemma}
\begin{proof}
	Write both $ V $ and $ H $ as direct sums of their Sylow $ p $-subgroups, and note that any map $ V\to H $ is entirely determined by its restriction to $ V_p $ for each prime $ p $. Also note that $ \calf(V_p)\subset H_p $. Suppose $ G_p $ is a $ p $-group of type $ \lambda $. Since $ V=R^n $, where $ R=\Z/a\Z $, we see that $ V_p=R_p^n=(\Z/p^{2\lambda_1}\Z)^n $. Moreover, $ \calf|_{V_{\sigma,p}}=\calf|_{(V_p)_\sigma} $. Thus, it suffices to prove the lemma for $ G $ a finite abelian $ p $-group.
	
	Suppose $ G $ is of type $ \lambda $, where $ \lambda $ is the partition $ \lambda_1\geq\cdots\geq\lambda_r $. Then $ H=(\Z/p^{2\lambda_1}\Z)^r $. Recall from \Cref{sec: Moments I} that we have a surjection $ \Lambda:H\twoheadrightarrow G $ and that $ \Lambda\circ\calf=F $. Let $ \sigma\subset[n] $ be a set of cardinality less than $ w $. We have the following sequence of maps \[V_\sigma/pV_\sigma\to H/pH\to G/pG,\] where the first map is given by $ \calf $ and the second by $ \Lambda $. The composite map $ V_\sigma/pV_\sigma\to G/pG $ is given by $ F $ and is surjective since $ F $ is a code of distance $ w $. Because $ H/pH $ and $ G/pG $ are $ \F_p $-vector spaces of rank $ r $, it follows that the map $ V_\sigma/pV_\sigma\to H/pH $ must be surjective. By Nakayama's lemma, it follows that $ \calf:V_\sigma\to H $ must be surjective as well. Therefore, $ \calf V_\sigma= H $ for any such $ \sigma $. 
\end{proof}

The following lemma will be useful often. In short, it allows us to choose a suitable basis of $ H $ that works particularly nicely with the $ v_i $'s in $ V $ in tandem with $ \calf $ and $ C $. We will use what it says about codes along with the property of robustness to obtain a good lower bound on the number of $ E_{ij}(C,D,\calf) $ that are nonzero. 
\begin{lemma}[Lemma 3.4, \cite{Wood}]\label{good basis}
	Let $ H $ be a finite $ R $-module with Sylow $ p $-subgroup of type $ \lambda $, where $ \lambda $ is a partition with $ r $ parts. Suppose $ F\in\Hom(V,H) $ is a code of distance $ \delta n $, and let $ C\in\Hom(V,H^*) $. Then we can find $ A_1,\ldots,A_r\in H $ and $ B_1,\ldots,B_r\in H^* $ such that for every $ 1\leq i\leq r $ \[\#\{j\in[n]\;|\;Fv_j=A_i\ \mathrm{and\ }Cv_j=B_i\}\geq\delta n/|H|^2,\] and after the projection to the Sylow $ p $-subgroup $ H_p $ of $ H $, the elements $ A_1,\ldots,A_r $ generate $ H_p $. 
\end{lemma}

\begin{lemma}[Quadratically many nonzero coefficients for robust $ (C,D) $]\label{quadratically many}
	Let $ G $ be a finite abelian group, and let $ P $ be the set of primes dividing $ |G| $. If $ \calf\in\Hom(V,H) $ is a code of distance $ \delta n $, and if $ (C,D)\in\Hom(V,G^*)\times\Sym^2 H^* $ is robust for $ \calf $, then there are at least $ \gamma\delta n^2/(2|H|^2|P|) $ pairs $ (i,j) $ with $ i\leq j $ such that $ E_{ij}(C,D,\calf)\neq0 $.	
\end{lemma}
\begin{proof}
	For each $ p\in P $, let $ H_p $ denote the Sylow $ p $-subgroup of $ H $, as usual. By taking the transpose of the surjection $ \Lambda:H\twoheadrightarrow G $, there is an inclusion of $ G^* $ into $ H^* $. Postcomposing with $ \Lambda^t $ gives us an element $ \Lambda^tC\in\Hom(V,H^*) $. With this in mind, apply \Cref{good basis} to $ H $, $ p $, $ \calf $, and $ \Lambda^tC $, and consider the resulting $ A_i(p) $ and $ B_i(p) $. Let \begin{equation}\label{tau}
		\tau_i(p)=\{j\in[n]\;|\;\calf v_j=A_i(p)\ \mathrm{and\ }\Lambda^tCv_j=B_i(p)\},
	\end{equation} and let $ \tau(p)=\bigcup_i\tau_i(p) $. Let $ V_p $ be the submodule generated by the $ v_j $ for $ j\in\tau(p) $ so that $ \calf V_p $ in the projection to $ H_p $ is the entirety of $ H_p $. 
	
	Next, let $ W $ be the submodule of $ V $ generated by the $ V_p $ for all $ p\in P $; note that $ \calf W=H $. Now, we apply \Cref{weak cor} to $ W $. Since $ (C,D) $ is robust for $ \calf $, we have
	\[S:=\#\{i\in[n]\;|\;t(\phi_{C,\Lambda\calf}(w),\phi_{\Lambda\calf,C}(v_i))+\ev(D,{\Omega\calf (w)\odot \Omega\calf (v_i)})\neq0\textrm{ for some }w\in W\}\geq\gamma n.\] If $ v_i $ pairs nontrivially with $ w\in W $, then it must pair nontrivially with an element of one of the submodules generating $ W $, implying that \begin{align*}
		\sum_{p\in P}&\#\{i\in[n]\;|\;t(\phi_{C,\Lambda\calf}(x),\phi_{\Lambda\calf,C}(v_i))+\ev(D,{\Omega\calf (x)\odot \Omega\calf (v_i)})\neq0\textrm{ for some }x\in V_p\}\geq S.
	\end{align*} Hence, for some $ p\in P $, \[\#\{i\in[n]\;|\;t(\phi_{C,\Lambda\calf}(x),\phi_{\Lambda\calf,C}(v_i))+\ev(D,{\Omega\calf (x)\odot \Omega\calf (v_i)})\neq0\textrm{ for some }x\in V_p\}\geq\gamma n/|P|.\] For that particular $ p $, it follows that \[\#\{i\in[n]\;|\;t(\phi_{C,\Lambda\calf}(v_j),\phi_{\Lambda\calf,C}(v_i))+\ev(D,{\Omega\calf (v_j)\odot \Omega\calf (v_i)})\neq0\ \mathrm{for\ some\ }j\in\tau(p)\}\geq\gamma n/|P|.\] By \eqref{tau}, for any $ j\in\tau(p) $, there are at least $ \delta n/|H|^2 $ indices $ k\in\tau(p) $ such that $ \calf v_k=\calf v_j $ and $ \Lambda^tC v_k=\Lambda^tCv_j $. Moreover, note that $ e(\calf(u),\Lambda^t C(v))=C(v)(\Lambda\calf(u))=e(\Lambda\calf(u),C(v)) $ for any $ u,v\in V $. Thus, for all indices $ k\in\tau(p) $ such that $ \calf v_k=\calf v_j $ and $ \Lambda^tC v_k=\Lambda^tCv_j $, we have \begin{align*}
		t(\phi_{C,\Lambda\calf}(v_j),\phi_{\Lambda\calf,C}(v_i))&+\ev(D,{\Omega\calf v_j\odot \Omega\calf v_i})\\
		&=t(\phi_{C,\Lambda\calf}(v_k),\phi_{\Lambda\calf,C}(v_i))+\ev(D,{\Omega\calf (v_k)\odot \Omega\calf (v_i)}).
	\end{align*} For all $ i< j $, notice that \[E_{ij}(C,D,\calf)=t(\phi_{C,\Lambda\calf}(v_j),\phi_{\Lambda\calf,C}(v_i))+\ev(D,{\Omega\calf (v_j)\odot \Omega\calf (v_i)})\] and also that \[2E_{ii}(C,D,\calf)=t(\phi_{C,\Lambda\calf}(v_i),\phi_{\Lambda\calf,C}(v_i))+\ev(D,{\Omega\calf (v_i)\odot \Omega\calf (v_i)}).\] Hence, there must be at least $ \gamma\delta n^2/(2|H|^2|P|) $ pairs $ (i,j) $ such that $ E_{ij}(C,D,\calf)\neq0 $.
\end{proof}

We next study the number of nonzero coefficients for weak $ (C,D) $. To do so, we will consider the pairs $ (C,D) $ such that all of the coefficients are 0 (e.g., $ (0,0) $ is one such pair, but there are many others). We will first need to adopt a more equivariant view of the coefficients $ E_{ij}(C,D,\calf) $. The evaluation map $ e:G^*\otimes G\to R $ gives rise to a natural map \[\Hom(V,G)\otimes\Hom(V,G^*)\simeq(V^*\otimes G)\otimes(V^*\otimes G^*)\to V^*\otimes V^*,\] and we may further compose with the quotient $ V^*\otimes V^*\to\Sym^2V^* $ to get a map \[\Phi:\Hom(V,G)\otimes\Hom(V,G^*)\to V^*\otimes V^*\to\Sym^2V^*.\] For $ \calf\in\Hom(V,H) $, composing with $ \Lambda $ gives us a map $ \Lambda\calf\in\Hom(V,G) $. 

Likewise, the evaluation $ \ev:\Sym_2 H\otimes\Sym^2 H^*\to R $ induces a natural map \[\Psi:\Hom(\Sym_2V,\Sym_2 H)\otimes\Sym^2 H^*\simeq(\Sym_2V)^*\otimes\Sym_2H\otimes\Sym^2H^*\to(\Sym_2V)^*.\] Recall that $ (\Sym_2V)^* $ is naturally isomorphic to $ \Sym^2V^* $. Given $ \calf\in\Hom(V,G) $, we get an element of $ \Hom(\Sym_2V,\Sym_2 H) $ by considering $ \Omega\calf\otimes\Omega\calf $. Taking the sum of the aforementioned maps $ \Phi $ and $ \Psi $ (not the direct sum), given $ \calf\in\Hom(V,H) $, we can construct a map \[m_\calf:\Hom(V,G^*)\oplus\Sym^2 H^*\to\Sym^2V^*\] given by \[m_\calf(C,D)=\Phi(\Lambda\calf\otimes C)+\Psi((\Omega\calf\otimes\Omega\calf)\otimes D).\] We see that $ m_\calf $ can be written explicitly as \begin{equation}\label{mF}
	\begin{split}
		m_\calf(C,D)=&\sum_{i=1}^n\sum_{j=i+1}^n\Big(e(C(v_j),\Lambda\calf(v_i))+e(C(v_i),\Lambda\calf(v_j))+\ev(D,{\Omega\calf(v_i)\odot\Omega\calf(v_j)})\Big)v_i^*v_j^*\\
		&\qquad\qquad+\sum_{i=1}^n\Big(e(C(v_i),\Lambda\calf(v_i))+\ev(D,\Omega\calf(v_i)\otimes\Omega\calf(v_i))\Big)(v_i^*)^2.
	\end{split}
\end{equation}
We remark that the definition of $ m_\calf $ is canonical and basis independent. In particular, the above could be rewritten by replacing the $ v_i $'s with the elements of any other basis of $ V $. Note that the elements in the kernel of $ m_\calf $ are precisely the $ (C,D) $ for which all of the coefficients $ E_{ij}(C,D,\calf) $ are zero. We call the $ (C,D)\in\ker(m_\calf) $ \emph{special} for $ \calf $. 

	Recall that the data of the pairing on $ \hG $ is equivalent to an element $ A $ of $ \Sym_2H $ whose precise definition is given in \Cref{sec: Moments I}.
	
	\begin{lemma}\label{special count}
		Given $ \calf\in\Hom(V,H) $ with $ \calf V=H $, there are $ {|\Sym^2 H|}/{|G|} $ special $ (C,D) $ for $ \calf $. Moreover, if $ (C,D) $ is special for $ \calf\in\Hom(V,H) $, then $ D(-A)=0 $.
	\end{lemma}
	\begin{proof}
		Because everything in sight can be written as a direct sum of Sylow $ p $-subgroups, we can assume without loss of generality that $ G $ is a $ p $-group of type $ \lambda $ (and, accordingly, that $ R=\Z/p^{2\lambda_1}\Z $). Suppose $ \lambda $ has $ r $ parts, and recall that $ H=R^r $. Because $ \calf $ is a surjection, there exist $ w_1,\ldots,w_r\in V $ that can be completed to a free $ R $-module basis $ w_1,\ldots,w_n $ of $ V $ such that $ \calf(w_i)=h_i $ and $ \calf(w_i)=0 $ for $ i>r $ (recall $ h_i $ is the $ i $th standard basis vector in $ H $). With respect to this basis, $ m_\calf $ takes $ (C,D) $ to \begin{align*}
			&\sum_{i=1}^n\sum_{j=i+1}^n\Big(e(C(w_j),\Lambda\calf(w_i))+e(C(w_i),\Lambda\calf(w_j))+\ev(D,{\Omega\calf(w_i)\odot\Omega\calf(w_j)})\Big)w_i^*w_j^*\\
			&\qquad\qquad+\sum_{i=1}^n\Big(e(C(w_i),\Lambda\calf(w_i))+\ev(D,\Omega\calf(w_i)\otimes\Omega\calf(w_i))\Big)(w_i^*)^2.
		\end{align*} 
		
		We compute the coefficients of each $ w_i^*w_j^* $ explicitly. For $ 1\leq i<j\leq r $, the coefficient on $ w_i^*w_j^* $ is \[e(C(w_j),p^{2\lambda_1-\lambda_i}h_i)+e(C(w_i),p^{2\lambda_1-\lambda_j}h_j)+p^{2\lambda_1-\lambda_i-\lambda_j}\ev(D,{h_i\odot h_j}).\] Recall that $ D $ has a unique upper-triangular representative $ D_0\in H^*\otimes H^* $ with respect to the basis $ h_1^*,\ldots,h_r^* $ and that we may do all our computations with respect to this representative element $ D_0 $. Since $ D_0 $ is upper-triangular with respect to the basis $ h_1^*,\ldots,h_r^* $, the coefficient of $ w_i^*w_j^* $ for $ 1\leq i<j\leq r $ is \[C_{ji}+C_{ij}+p^{2\lambda_1-\lambda_i-\lambda_j}D_{ij},\] where $ C_{ij} $ denotes $ C(w_i)(p^{2\lambda_1-\lambda_j}h_j) $ and $ D_{ij} $ denotes $ D_0(h_i)(h_j) $ (recall that $ C $ takes elements of $ V $ to $ G^*=\Hom(\Lambda H,R) $). Note that $ p^{\lambda_j}C_{ij}=0\in R $, i.e., $ C_{ij}\equiv0$ mod $ p^{2\lambda_1-\lambda_j} $. For $ i=j $, the coefficient on $ w_i^*w_j^* $ is $ C_{ii}+p^{2\lambda_1-2\lambda_i}D_{ii}. $ If $ 1\leq i\leq r<j $, the coefficient on $ w_i^*w_j^* $  is $ C_{ji}. $ 
		
		For $ (C,D) $ to be special for $ \calf $, all of these coefficients must be 0 in $ R $. Setting all of these coefficients to 0 gives us the following system of equations 
		\begin{equation}\label{special system}
			\begin{cases}
				C_{ji}+C_{ij}+p^{2\lambda_1-\lambda_i-\lambda_j}D_{ij}=0 & \textrm{for }1\leq i<j\leq r;\\
				C_{ii}+p^{2\lambda_1-2\lambda_i}D_{ii}=0 & \textrm{for }1\leq i\leq r;\\
				C_{ji}=0 &\textrm{for }1\leq i\leq r<j\leq n.
			\end{cases}
		\end{equation}
		modulo $ p^{2\lambda_1} $, which the entries of $ C $ and $ D $ satisfy if and only if $ (C,D) $ is special. We count the pairs $ (C,D) $ satisfying this system of equations. Recall that $ G=\Lambda H $ and that $ (C,D)\in\Hom(V,G^*)\times\Sym^2H^* $. Hence, the possible $ C\in\Hom(V,G^*) $ are given by choosing $ C_{ij} $ to be any multiple of $ p^{2\lambda_1-\lambda_j} $ in $ R $ for each $ 1\leq i\leq n$ and $ 1\leq j\leq r $ (recall that $ C_{ij}\equiv 0 $ mod $ p^{2\lambda_1-\lambda_j} $). Because $ D $ is determined by its unique upper-triangular representative $ D_0 $, the possible $ D $ are given by choosing the entries $ D_{ij}\in R $ for $ 1\leq i\leq j\leq r $. The $ D_{ij} $'s determine $ D_0 $, which determines $ D $. Since $ D\in\Sym^2H^* $, the $ D_{ij} $'s can be any elements of $ R $.

		\Cref{special system} determines $ C_{ji} $ for $ 1\leq i\leq r<j\leq n $. In order for there to be solutions $ D_{ij} $, we need $ C_{ji}+C_{ij}\equiv0 $ mod $ p^{2\lambda_1-\lambda_i-\lambda_j} $. This holds because $ C_{ij}\equiv0 $ modulo $ p^{2\lambda_1-\lambda_j} $. Thus, for $ 1\leq i, j\leq r $, choose the $ C_{ji} $'s freely; there are $ p^{\lambda_i} $ choices for each $ C_{ji} $. \Cref{special system} then determines the $ D_{ij} $ modulo $ p^{\lambda_i+\lambda_j} $ for all $ 1\leq i\leq j\leq r $. It follows that there are $ p^{2\lambda_1-\lambda_i-\lambda_j} $ choices for each $ D_{ij}\in R $ such that \eqref{special system} is satisfied. Therefore, there are \[\prod_{1\leq i\leq r}p^{r\lambda_i}\cdot\prod_{1\leq i\leq j\leq r}p^{2\lambda_1-\lambda_i-\lambda_j}=\frac{|G|^rp^{2\lambda_1{{r+1}\choose{2}}}}{p^{\sum_{1\leq i\leq j\leq r}(\lambda_i+\lambda_j)}}=\frac{|G|^r|\Sym^2H|}{|G|^{r+1}}=\frac{|\Sym^2H|}{|G|}\] special $ (C,D) $, as desired.
		
		To prove the second part of the lemma, we again reduce to the case to where $ G $ is a $ p $-group of type $ \lambda $, where $\lambda$ is the partition $ \lambda_1\geq\cdots\geq\lambda_r $, and $ R=\Z/p^{2\lambda_1}\Z $. Recall $ A\in\Sym_2 H $ is the element \[\sum_{1\leq i<j\leq r}p^{2\lambda_1-\lambda_j}a_{ij}h_i\odot h_j+\sum_{i=1}^rp^{2\lambda_1-\lambda_i}a_{ii}h_i\otimes h_i.\] Then we may compute $ D(-A)=\ev(D,-A) $ with respect to the unique upper-triangular representative $ D_0 $ of $ D $ from above. By hypothesis, $ (C,D) $ is special for $ \calf $. Hence, for $ 1\leq i<j\leq r $, we have  \[C_{ji}+C_{ij}+p^{2\lambda_1-\lambda_i-\lambda_j}D_{ij}=0\mod p^{2\lambda_1},\] and for $ 1\leq i=j\leq r $, we have $ C_{ii}+p^{2\lambda_1-2\lambda_i}D_{ii}=0\mod p^{2\lambda_1} $. Since $ C_{ij}\equiv0 $ mod $ p^{2\lambda_1-\lambda_j} $, we have that $ D_{ij}\equiv0 $ mod $ p^{\lambda_j} $ for all $ i\leq j $. Thus, \begin{align*}
			\ev(D,A)&=\sum_{1\leq i<j\leq r}p^{2\lambda_1-\lambda_j}a_{ij}(D_0(h_i)(h_j)+D_0(h_j)(h_i))+\sum_{i=1}^rp^{2\lambda_1-\lambda_i}a_{ii}D_0(h_i)(h_i)\\
			&=\sum_{1\leq i<j\leq r}p^{2\lambda_1-\lambda_j}a_{ij}D_{ij}+\sum_{i=1}^rp^{2\lambda_1-\lambda_i}a_{ii}D_{ii}\\
			&=0\in R,
		\end{align*} where the second equality follows from the fact that $ D_0 $ is upper-triangular. 
	\end{proof}
	
	Now, we will use \Cref{special count} to prove that as long as $ (C,D) $ is not special for $ \calf $ there are linearly many nonzero coefficients $ E_{ij}(C,D,\calf) $. 
	\begin{lemma}[Linearly many nonzero coefficients for non-special $ (C,D) $]\label{linearly many}
		Let $ \calf\in\Hom(V,H) $ be a code of distance $ \delta n $, and suppose $ (C,D)\in\Hom(V,G^*)\times\Sym^2 H^*$ is not special for $ \calf $ so that $ (C,D)\not\in\ker(m_{\calf}) $. Then there are at least $ \delta n/2 $ pairs $ (i,j)\in[n]^2 $ with $ i\leq j $ such that $ E_{ij}(C,D,\calf)\neq0 $. 
	\end{lemma}
	\begin{proof}
		Suppose for a contradiction that there are fewer than $ \delta n/2 $ pairs. Let $ \pi $ be the set of pairs $ (i,j)\in[n]^2 $ with $ i\leq j $ and $ E_{ij}(C,D,\calf)\neq0 $. By assumption, $ |\pi|<\delta n/2 $. Let $ \sigma $ be the set of all indices $ a\in[n] $ occurring in a pair $ (i,j)\in\pi $. Immediately, we see that $ |\sigma|<\delta n $, and if $ (i,j) $ is a pair with $ i $ or $ j $ not in $ \sigma $, then $ E_{ij}(C,D,\calf) $ must be 0. 
		
		We will use this setup to give a lower bound on the size of $ \im(m_\calf) $. Recall the definition of $ m_\calf $ from \eqref{mF}: $ (C,D) $ is sent to the element of $ \Sym^2V^* $ given by \begin{align*}&\sum_{i=1}^n\sum_{j=i+1}^n\Big(e(C(v_j),\Lambda\calf(v_i))+e(C(v_i),\Lambda\calf(v_j))+\ev(D,{\Omega\calf(v_i)\odot\Omega\calf(v_j)})\Big)v_i^*v_j^*\\
			&\qquad\qquad+\sum_{i=1}^n\Big(e(C(v_i),\Lambda\calf(v_i))+\ev(D,\Omega\calf(v_i)\otimes\Omega\calf(v_i))\Big)(v_i^*)^2.
		\end{align*} Compose $ m_\calf $ with the quotient map taking $ v_i^*v_j^*\mapsto0 $ if both $ i,j\in\sigma $, and call the resulting map $ m_\calf':\Hom(V,G^*)\times\Sym^2 H^*\to Z $, where $ Z $ denotes our new quotient space. Under $ m_\calf' $, we see that $ (C,D)\mapsto0 $ (by hypothesis $ (C,D)\not\in\ker(m_\calf) $ since it is not special). We will arrive at a contradiction by proving that $ \#\im(m_\calf)=\#(\Hom(V,G^*)\times\Sym^2 H^*)/\#\ker(m_\calf) $ divides $ \#\im(m_\calf') $. Then, since $ \#\im(m_\calf')|\#\im(m_\calf) $, this will force $ \ker(m_\calf)=\ker(m_\calf') $, which is a contradiction. By \Cref{special count}, \[|\im(m_\calf)|=\frac{|\Hom(V,G^*)\times\Sym^2 H^*|}{|\ker(m_\calf)|}=|G|^{n+1}.\]
		
		As in the proof of \Cref{special count}, we can prove $ |G|^{n+1} $ divides $ |\im(m_\calf')| $ by reducing to the case where $ G $ is a $ p $-group of type $ \lambda $, where $ \lambda $ is the partition $ \lambda_1\geq\cdots\geq\lambda_r $. Accordingly, assume that $ R=\Z/p^{2\lambda_1}\Z $.
		
		Because $ \calf $ is a code of distance $ \delta n $, and because $ |\sigma|<\delta n $, we have that $ \calf V_\sigma=H $. Hence $ \calf|_{V_\sigma} $ is a code of some positive distance. Apply \Cref{good basis} to $ \calf|_{V_\sigma} $ to get $ \tau\subset[n]\setminus\sigma $ such that $ |\tau|=r $ and $ \calf v_i $ generate $ H $ for $ i\in\tau $. Let $ W $ be the $ R $-submodule of $ V $ generated by the $ v_i $ for $ i\in\tau $. Let $ h_1\ldots,h_r $ be generators for $ H $ with relations $ p^{2\lambda_1}h_i=0 $, and let $ h_i^* $ denote the corresponding dual basis for $ H^* $ so that $ h_i^*(h_j)=\delta_{ij} $. For $ 1\leq j\leq r $, suppose $ w_j\in W $ is such that $ \calf w_j=h_j $, and let $ U\subset W $ be the submodule of $ W $ generated by the $ w_j $. Now, we have maps \[U/pU\to W/pW\to H/pH,\] where the second map is given by $ \calf $. Each of the above spaces is an $ \F_p $-vector space of rank at most $ r $, exactly $ r $, and exactly $ r $, respectively. Because the composite map $ U/pU\to H/pH $ is a surjection, our considerations on the rank of the spaces tell us that the map $ U/pU\to W/pW $ is a surjection. Thus, Nakayama's lemma forces $ U=W $. The elements $ w_1,\ldots,w_r $ generate the free $ R $-module $ U $ of rank $ r $, so they must form a basis for $ W $ as a free $ R $-module. Thus, we have found a basis for $ W $ such that $ \calf(w_i)=h_i $. 
		
		Consider the following basis for $ V $, given by replacing the $ v_i $ for $ i\in\tau $ with the $ w_k $'s. Let $ \tau=\{\tau_1,\ldots,\tau_r\} $. For $ i\in\tau $, let $ z_i=w_{\tau_i} $, and for $ i\not\in\tau $ set $ z_i=v_i $. Because the $ w_j $'s are a basis for $ W $, it follows that $ z_1,\ldots,z_n $ forms a basis for $ V $; denote the corresponding dual basis by $ z_1^*,\ldots,z_n^* $, where $ z_i^*(z_j)=\delta_{ij} $.
		
		View $ \calf\in\Hom(V,H) $ as an element of $ V^*\otimes H $ and write \[\calf=\sum_{\substack{1\leq i\leq n\\1\leq j\leq r}}f_{ij}z_i^*\otimes h_j\] so that $ \calf(z_\ell)=\sum_{1\leq k\leq r}f_{\ell k}h_k $. Because $ m_\calf $ was defined equivariantly, we see that $ m_\calf(C,D) $ is given by \begin{align*}
			&\sum_{i=1}^n\sum_{j>i}^n\Big(e(C(z_j),\Lambda\calf(z_i))+e(C(z_i),\Lambda\calf(z_j))+\ev(D,{\Omega\calf(z_j)\odot\Omega\calf(z_i)})\Big)z_i^*z_j^*\\
			&\quad\quad\quad+\sum_{i=1}^n\Big(e(C(z_i),\Lambda\calf(z_i))+\ev(D,\Omega\calf(z_i)\otimes\Omega\calf(z_i))\Big)(z_i^*)^2
		\end{align*} which we may rewrite as \begin{align*}
			&\sum_{i=1}^n\sum_{j>i}^n\Big(e(C(z_j),\textstyle\sum_{1\leq k\leq r}\Lambda(f_{ik}h_k))+e(C(z_i),\textstyle\sum_{1\leq k\leq r}\Lambda(f_{jk}h_k))+\ev(D,{\Omega\calf(z_j)\odot\Omega\calf(z_j)})\Big)z_i^*z_j^*\\
			&\quad\quad\quad+\sum_{i=1}^n\Big(e(C(z_i),\textstyle\sum_{1\leq k\leq r}\Lambda(f_{ik}h_k))+\ev(D,\Omega\calf(z_i)\otimes\Omega\calf(z_i))\Big)(z_i^*)^2.
		\end{align*}
		
		Hence, 
		\begin{align*}
			m_\calf(z_\ell^*\otimes h_m^*,0)&=\sum_{i=1}^n\sum_{j=1,j\neq i}^n\Big(e(z_\ell^*(z_j)\otimes h_m^*,\textstyle\sum_{1\leq k\leq r}p^{2\lambda_1-\lambda_k}f_{ik}h_k)\Big)z_i^*z_j^*\\
			&\quad\quad\quad+\displaystyle\sum_{i=1}^n\Big(e(z_\ell^*(z_i)\otimes h_m^*,\textstyle\sum_{1\leq k\leq r}p^{2\lambda_1-\lambda_k}f_{ik}h_k)\Big)(z_i^*)^2.\\
			&=\sum_{i=1}^n\sum_{j=1,j\neq i}^np^{2\lambda_1-\lambda_m}f_{im}z_\ell^*(z_j)(z_i^*z_j^*)+\sum_{i=1}^np^{2\lambda_1-\lambda_m}f_{im}z_\ell^*(z_i)(z_i^*)^2\\
			&=\sum_{i=1}^np^{2\lambda_1-\lambda_m}f_{im}z_i^*z_\ell^*.
		\end{align*}
		
		Now, further quotient $ Z $ by the submodule generated by the $ v_i^*v_j^* $ such that neither $ i $ nor $ j $ are in $ \tau $. Call the resulting space $ Z' $ and denote the composite $ \Hom(V,G^*)\times\Sym^2 H^*\to Z' $ by $ m_\calf'' $. Because $ \calf(z_t)=\sum_{1\leq k\leq r}f_{t k}h_k $ and $ \calf (z_{\tau_i})=h_i $, it follows that $ f_{\tau_ij}=\delta_{ij} $. Therefore, if $ \ell\not\in\tau $, \begin{equation}\label{mF''}
			m_\calf''(z_\ell^*\otimes h_m^*,0)=\sum_{i\in\tau}p^{2\lambda_1-\lambda_m}f_{im}z_i^*z_\ell^*=\sum_{i=1}^rp^{2\lambda_1-\lambda_m}f_{\tau_im}z_{\tau_i}^*z_\ell^*=p^{2\lambda_1-\lambda_m}w_m^*z_\ell^*.
		\end{equation} The element $ p^{2\lambda_1-\lambda_m}w_m^*z_\ell^* $ has order $ p^{\lambda_m}$, implying that $ \im(m_\calf'') $ has a subgroup of order $ |G|^{n-r} $, as we can take any $ \ell\in[n]\setminus\tau $ and $ 1\leq m\leq r $. 
		
		Continuing, further quotient $ Z' $ by the submodule generated by the $ z_i^*z_j^* $ such that one of $ i $ or $ j $ is not in $ \tau $. Denote the resulting space by $ Z'' $ and map by $ m_\calf''' $. Note that the only basis elements $ z_i^*z_j^* $ of $ \Sym^2V^* $ whose images are nonzero in $ Z'' $ are those with both $ i,j\in\tau $. Also, note that the subgroup of size $ |G|^{n-r} $ we identified previously is sent to 0 under $ m_\calf''' $. Suppose $ 1\leq y\leq x\leq r $ and let $ D_{xy}:H\to H^* $ be the map $ h_x^*\otimes h_y^* $ taking $ h_x\mapsto h_y^* $ and all other $ h_j\mapsto0 $. Consider the image of $ D_{xy} $ in $ \Sym^2 H^*$, which we will denote using $ [D_{xy}]=[h_x^*\otimes h_y^*] $. 
		
		Suppose $ \ell=\tau_s\in\tau $. We see that \begin{equation}\label{mF'''}
			\begin{split}
				m_\calf'''&(z_\ell^*\otimes h_m^*,[D_{xy}])\\
				&=\sum_{i\in\tau}p^{2\lambda_1-\lambda_m}f_{im}z_i^*z_\ell^*+\sum_{\substack{i,j\in\tau\\i< j}}\ev\left([h_x^*\otimes h_y^*],{\left(\textstyle\sum_k p^{\lambda_1-\lambda_k}f_{jk}h_k\right)\odot\left(\textstyle\sum_k p^{\lambda_1-\lambda_k}f_{ik}h_k\right)}\right)z_i^*z_j^*\\
				&\quad\quad+\sum_{i\in\tau}\ev\left([h_x^*\otimes h_y^*],{\left(\textstyle\sum_k p^{\lambda_1-\lambda_k}f_{ik}h_k\right)\otimes\left(\textstyle\sum_k p^{\lambda_1-\lambda_k}f_{ik}h_k\right)}\right)(z_i^*)^2\\
				&=\sum_{i=1}^rp^{2\lambda_1-\lambda_m}f_{\tau_i m}w_i^*z_\ell^*+\sum_{\substack{i,j\in\tau\\i< j}}p^{2\lambda_1-\lambda_x-\lambda_y}(f_{jx}f_{iy}+f_{jy}f_{ix})z_i^*z_j^*+\sum_{i\in\tau}p^{2\lambda_1-2\lambda_i}f_{ix}f_{iy}(z_i^*)^2\\
				&=p^{2\lambda_1-\lambda_m}w_m^*z_\ell^*+p^{2\lambda_1-\lambda_x-\lambda_y}w_{x}^*w_y^*.
			\end{split}
		\end{equation} For all $ m $ such that $ m\geq s $, set $ x=m $ and $ y=s $. Then \[m_\calf'''(z_\ell^*\otimes h_m^*,[D_{ms}])=(p^{2\lambda_1-\lambda_m}+p^{2\lambda_1-\lambda_m-\lambda_s})w_m^*w_s^*=p^{2\lambda_1-\lambda_m-\lambda_s}(p^{\lambda_s}+1)w_m^*w_s^*.\] Since $ p^{\lambda_s}+1 $ is invertible in $ R $, this tells us that $ \im(m_\calf''') $ has a subgroup of size \[\prod_{1\leq s\leq m\leq r}p^{\lambda_m+\lambda_s}=p^{\sum_{1\leq s\leq m\leq r}\lambda_m+\lambda_s}=|G|^{r+1}.\] Therefore, we conclude that $ \# G^{n+1}=\# G^{r+1}\cdot\# G^{n-r}|\#\im(m_\calf'')|\#\im(m_\calf') $. By our prior explanation, this completes the proof of the lemma. 
	\end{proof}

	\section{Obtaining the moments III: A good bound for codes}\label{sec: Moments III}
	We combine the results proved in the previous section to bound \[\Prob(FM=0\mathrm{\ and\ }F^t(\p_M)=\phG)=\Prob(\Lambda\calf M=0\mathrm{\ and\ }\Omega\calf M\calf^t\Omega^t=A)\] when $ F $ is a code and $ M $ is a random matrix. If $ M $ is a random symmetric matrix with integer entries $ m_{ij} $ and $ b $ is an integer, we say $ M $ is $ \alpha $-\emph{balanced} mod $ b $ if for any prime divisor $ p $ of $ b $ and $ t\in\Z/p\Z $, the probability $ \Prob(m_{ij}\equiv t\ \mathrm{mod\ }p)\leq1-\alpha $. 
	\begin{lemma}\label{bound for codes}
		Let $ 0<\alpha<1 $, let $ \delta>0 $, and let $ G $ be a finite abelian group. Let $ a $, $ H $, and $ R $ be defined as in \Cref{sec: Moments I}. Then there exists some $ c,K>0 $ with the following significance.
		
		Let $ M $ be a random symmetric $ n\times n $ matrix whose entries $ m_{ij}\in\Z/a\Z $ for $ 1\leq i\leq j\leq n $ are independent and $ \alpha $-balanced mod $ a $. Suppose $ F\in\Hom(V,G) $ is a code of distance $ \delta n $, and let $ \calf\in\Hom(V,H) $ be a lift of $ \calf $ to $ H $. For all $ n $, we have that \[|\Prob(\Lambda\calf M=0\mathrm{\ and\ }\Omega\calf M\calf^t\Omega^t=A)-|G|^{-(n+1)}|\leq\frac{K\exp(-cn)}{|G|^n}.\]
	\end{lemma}
	To prove the above, we will need the following result due to Wood.
	\begin{lemma}[Lemma 4.2, \cite{Wood}]\label{expected value bound}
		Suppose $ \xi $ is a primitive $ b $th root of unity, and let $ z $ be a random variable taking values in $ \Z/b\Z $ such that $ z $ takes each value with probability at most $ 1-\alpha $. Then $ |\E(\xi^z)|\leq\exp(-\alpha/b^2) $. 
	\end{lemma}
	\begin{proof}[Proof of \Cref{bound for codes}]
		Recall that \[\Prob(\Lambda\calf M=0\mathrm{\ and\ }\Omega\calf M\calf^t\Omega^t=A)=\frac{1}{|G|^{n}|\Sym^2 H|}\sum_{\substack{C\in\Hom(V,G^*)\\ D\in\Sym^2 H^*}}\E(\zeta^{C(\Lambda\calf M)+D(\Omega\calf M\calf^t\Omega^t-A)}),\] where $ \zeta $ is a primitive $ a $th root of unity. Further, by \Cref{lift of a code is a code}, we know that $ \calf $ is a code of distance $ \delta n $, so we may apply our results from the previous section.
		
		To prove the theorem, we break the sum into three pieces (we will choose $ 0<\gamma<\delta $ later):\begin{enumerate}
			\item $ (C,D) $ is special for $ \calf $;
			\item $ (C,D) $ is not special and weak for $ \calf $;
			\item $ (C,D) $ is robust for $ \calf $. 
		\end{enumerate}
		
		For a fixed $ \calf $, there are $ |\Sym^2 H|/|G| $ special pairs $ (C,D) $ by \Cref{special count}. We may write \[\E(\zeta^{C(\Lambda\calf M)+D(\Omega\calf M\calf^t\Omega^t-A)})=\E(\zeta^{D(-A)})\E(\zeta^{C(\Lambda\calf M)+D(\Omega\calf M\calf^t\Omega^t)}).\] If $ (C,D) $ is special for $ \calf $, recall from \Cref{special count} that \[C(\Lambda\calf M)+D(\Omega\calf M\Omega^t\calf^t)=D(-A)=0\in R.\] Therefore, for $ (C,D) $ special for $ \calf $, \[\E(\zeta^{C(\Lambda\calf M)+D(\Omega\calf M\calf^t\Omega^t-A)})=1\] for any $ M $. 
		
		For the second case, we will rely on the fact that there are not too many weak pairs $ (C,D) $ in conjunction with our bound from \Cref{linearly many}. Recall from \Cref{weak estimate} that for a given $ \calf $ there are at most \[C_G{{n}\choose{\lceil\gamma n\rceil-1}}|G|^{\gamma n}\] weak $ (C,D)\in\Hom(V,G^*)\times\Sym^2 H^*$. Now, factor the expected value as in \eqref{factored expected value} to get \begin{align*}
			\E(&\zeta^{C(\Lambda\calf M-B)+D(\Omega\calf M\calf^t\Omega^t-A)})=\E(\zeta^{D(-A)})\prod_{1\leq i<j\leq n}\E(\zeta^{E_{ij}(C,D,\calf)m_{ij}})\prod_{1\leq i\leq n}\E(\zeta^{E_{ii}(C,D,\calf)m_{ii}}).
		\end{align*}
		
		Recall that we are assuming $ (C,D) $ is not special for $ \calf $. \Cref{linearly many} tells us that at least $ \delta n/2 $ of the $ E_{ij}(C,D,\calf) $ are nonzero. Hence, if $ (C,D) $ is not special for $ \calf $, we conclude using \Cref{expected value bound} that \[|\E(\zeta^{C(\Lambda\calf M-B)+D(\Omega\calf M\calf^t\Omega^t-A)})|\leq\exp(-\alpha\delta n/(2a^2)).\]
		
		For the third case, we recall that given $ \calf $ and robust $ (C,D) $ for $ \calf $, \Cref{quadratically many} implies that we have at least $ \gamma\delta n^2/(2|H|^2) $ pairs $ (i,j) $ such that $ E_{ij}(C,D,\calf)\neq0 $. Thus, if $ (C,D) $ is robust for $ \calf $, then \[|\E(\zeta^{C(\Lambda\calf M-B)+D(\Omega\calf M\calf^t\Omega^t-A)})|\leq\exp(-\alpha\gamma\delta n^2/(2|H|a^2)).\]
		
		Putting all three cases together, we get \begin{align*}
			\bigg|&\Prob(\Lambda\calf M=0\mathrm{\ and\ }\Omega\calf M\calf^t\Omega^t=A)-\frac{1}{|G|^{n}|\Sym^2 H|}\sum_{(C,D)\;\mathrm{special}}\E(\zeta^{C(\Lambda\calf M)+D(\Omega\calf M\calf^t\Omega^t-A)})\bigg|\\
			&\leq\frac{|G|^{-n}}{|\Sym^2 H|} \sum_{(C,D)\;\mathrm{not\ special}}|\E(\zeta^{C(\Lambda\calf M)+D(\Omega\calf M\calf^t\Omega^t-A)})|\\
			&\leq\frac{|G|^{-n}}{|\Sym^2 H|}\bigg(C_G{{n}\choose{\lceil\gamma n\rceil-1}}|G|^{\gamma n}\exp(-\alpha\delta n/(2a^2))+|G|^n|\Sym^2 H|\exp(-\alpha\gamma\delta n^2/(2|H|^2a^2))\bigg).
		\end{align*}
		Let $ c>0 $ be such that $ c<\alpha\delta/(2a^2) $. Given $ \delta,\alpha,G,c $, we can choose $ \gamma $ sufficiently small so that \begin{align*}
			\bigg|\Prob(\Lambda\calf M=0&\mathrm{\ and\ }\Omega\calf M\calf^t\Omega^t=A)-\frac{|G|^{-n}}{|\Sym^2 H|}\sum_{(C,D)\;\mathrm{special}}\E(\zeta^{C(\Lambda\calf M)+D(\Omega\calf M\calf^t\Omega^t-A)})\bigg|\\
			&\leq\frac{1}{|G|^n}\left(C_G\exp(-cn)+\exp(\log(|G|)n-\alpha\gamma\delta n^2/(2|H|^2a^2))\right).
		\end{align*} Now, for fixed $ \alpha,G,\delta,\gamma,c $, the inequality \[\log(|G|)n-\alpha\gamma\delta n^2/(2|H|^2a^2)\leq-cn\] holds for $ n $ sufficiently large. Hence, we have \[\left|\Prob(\Lambda\calf M=0\mathrm{\ and\ }\Omega\calf M\calf^t\Omega^t=A)-|G|^{-(n+1)}\right|\leq\frac{(C_G+1)\exp(-cn)}{|G|^n}\] for $ n $ large enough. For the finitely many $ n $ such that the above is false, just increase the constant $ K $ in the statement of the result. 
	\end{proof}
	\section{Obtaining the moments IV: The structural properties of the surjections}\label{sec: Moments IV}
	In this section, we treat the case where $ F\in\Hom(V,G) $ is not a code. To do so, we use a division of the $ F $ based on the subgroups of $ G $ from \cite{Wood}. Let $ D $ be an integer with prime factorization $ \prod_ip_i^{e_i} $, and let $ \ell(D)=\sum_ie_i $. The following is an amalgamation of useful ideas and results developed by Wood in \cite{Wood} that we will apply later on in the next section.
	\begin{definition}
		The \emph{depth} of a homomorphism $ F\in\Hom(V,G) $ is the maximal positive $ D $ such that there exists $ \sigma\subset[n] $ with $ |\sigma|<\ell(D)\delta n $ such that $ D=[G:FV_\sigma] $. If no such $ D $ exists, the depth of $ F $ is 1. (When applying this definition, we always assume $ \delta<\ell(|G|)\inverse $.)	If the depth of $ F $ is 1, then note that $ F $ is a code of distance $ \delta n$: if $ F $ has depth 1, then for every $ \sigma\subset[n] $ with $ |\sigma|<\delta n $ we have $ FV_\sigma=G $ (otherwise, $ \ell([G:FV_\sigma])\geq1 $).
	\end{definition}
	Note that if the depth of $ F $ is $ D $, then $ D|\# G $. The following bounds the number of $ F $ of depth $ D $.
	\begin{lemma}[Estimate for $ F $ of depth $ D $; Lemma 5.2 in \cite{Wood}]\label{depth D estimate}
		There is a constant $ K $ depending on $ G $ such that if $ D>1 $, then the number of $ F\in\Hom(V,G) $ of depth $ D $ is at most \[K{{n}\choose{\lceil\ell(D)\delta n\rceil-1}}|G|^nD^{-n+\ell(D)\delta n}.\]
	\end{lemma}
	
	When we are working with the Laplacian of a random graph, we also need a bound on the number of $ F $ of depth $ D $, which is given by the following lemma from \cite{Wood}. If $ G $ is a finite abelian group with exponent $ b $, then so is $ G\oplus\Z/b\Z $. Thus, we may apply the definition of depth to a homomorphism $ F\in\Hom(V,G\oplus\Z/b\Z) $. 
	
	\begin{lemma}[Lemma 5.3 of \cite{Wood}]\label{sandpile depth}
		Let $ \pi_2:G\oplus\Z/b\Z\to\Z/b\Z $ denote projection onto the second factor. Given a finite abelian $ p $-group $ G $ of exponent $ b $, there is a constant $ K $, which depends on $ G $, such that if $ D>1 $, then the number of $ F\in\Hom(V,G\oplus\Z/b\Z) $ of depth $ D $ such that $ \pi_2(Fv_i)=1 $ for all $ i\in [n] $ is at most \[K{{n}\choose{\lceil\ell(D)\delta n\rceil}-1}|G|^n|D|^{-n+\ell(D)\delta n}.\]
	\end{lemma}
	
	For each $ F $ of depth $ D $, we use the following lemma to bound $ \Prob(FM=0\mathrm{\ and\ }F^t(\p_M)=\phG) $. In particular, because \[\Prob(FM=0\mathrm{\ and\ }F^t(\p_M)=\phG)\leq\Prob(FM=0)\] and because the $ F $ that are not codes only contribute to our error term, we can simply apply the bounds for $ \Prob(FM=0) $ due to Wood from \cite{Wood}.
	
	\begin{lemma}\label{bound probability for depth}
		Let $ \alpha,\delta,G,a $ be as in \Cref{bound for codes}. Let $ F\in\Hom(V,G) $ have depth $ D $, and suppose $ [G:FV]<D $ (e.g., this holds if $ FV=G $). Then there is a real number $ K $ with the following significance. For all $ M $ as in \Cref{bound for codes} and all $ n $, \[\Prob(FM=0)\leq Ke^{-\alpha(1-\ell(D)\delta)n}(|G|/D)^{-(1-\ell(D)\delta)n}.\]
	\end{lemma}
	\section{Obtaining the moments V: Oh yeah, it's all coming together}\label{sec: Moments V}
	We aggregate the work from the previous four sections to produce a universality result on (and the actual values of) the moments of the cokernels of random matrices with their pairings as well as the moments of sandpile groups of random graphs with their pairings (see the following two theorems, respectively). It will be helpful to recall the definition of an $ \alpha $-balanced random matrix from \Cref{sec: Moments III}.
	\begin{theorem}\label{random matrix moment}
		Suppose $ 0<\alpha<1 $, and let $ G $ be a finite abelian group such that $ \hG $ is equipped with a symmetric pairing. For $ c $ sufficiently small, there exists some $ K>0 $, depending on $ \alpha,G,c $, with the following significance. Let $ M $ be a random symmetric $ n\times n $ integral matrix, $ \alpha $-balanced mod $ |G| $, whose entries $ m_{ij}\in\Z $ for $ i\leq j $ are independent. Then \[|\E(\#\Sur^*(\cok(M),G))-|G|\inverse|\leq Ke^{-cn}.\]
	\end{theorem}
	\begin{proof}
		We skip several details in this proof, since the strategy is almost identical to (and, in many cases, simpler than) that of \Cref{sandpile moments}. Let $ b $ be the exponent of $ G $, and let $ a=b^2 $. Reduce the matrix $ M $ modulo $ a $ so that our notation agrees with that of previous sections. Recall that \[\E(\#\Sur^*(\cok(M),G))=\sum_{F\in\Sur(V,G)}\Prob(FM=0\mathrm{\ and\ }F^t(\p_M)=\phG).\] By \Cref{bound probability for depth} and \Cref{depth D estimate}, we have that \[\sum_{\substack{F\in\Sur(V,G)\\\tilde{F}\mathrm{\ not\ code\ of\ distance\ }\delta n}}\Prob(FM=0\mathrm{\ and\ }F^t(\p_M)=\phG)\leq Ke^{-cn}.\] \Cref{depth D estimate} also implies that \[\sum_{\substack{F\in\Sur(V,G)\\\tilde{F}\mathrm{\ not\ code\ of\ distance\ }\delta n}}|G|^{-n-1}\leq Ke^{-cn}.\] We can also show that \[\sum_{F\in\Hom(V,G)\setminus\Sur(V,G)}|G|^{-n-1}\leq K2^{-n}.\] Finally, applying \Cref{bound for codes}, we see that \[\sum_{\substack{F\in\Sur(V,G)\\\tilde{F}\mathrm{\ code\ of\ distance\ }\delta n}}\left|\Prob(FM=0\mathrm{\ and\ }F^t(\p_M)=\phG)-|G|^{-n-1}\right|\leq Ke^{-cn};\] putting this all together yields the desired result. 
	\end{proof}
	\begin{theorem}\label{sandpile moments}
		Let $ 0<q<1 $, and let $ G $ be a finite abelian group equipped with a symmetric pairing on its dual $ \hG=\Hom(G,\Q/\Z) $. Then there exist positive constants $ c,K $ (depending on $ G $) with the following significance. Suppose $ \Gamma\in G(n,q) $ is a random graph and $ S $ its sandpile group. Let $ L $ be the Laplacian of $ \Gamma $, and let $ \p_{\h{S}} $ denote the canonical symmetric pairing on $ \h{S} $. Then for all $ n $ we have that \[|\E(\#\Sur^*(S,G))-|G|\inverse|\leq Ke^{-cn}.\] 
	\end{theorem}
	\begin{proof}
		Suppose $ b $ is the exponent of $ G $, and let $ \mathfrak{R}=\Z/b\Z $. Let $ a=b^2 $, and let $ R=\Z/a\Z $. Let $ S_R=S\otimes R $ and $ S_\mathfrak{R}=S\otimes \mathfrak{R} $, and recall that $ \#\Sur(S,G)=\#\Sur(S\otimes \mathfrak{R},G)=\#\Sur(S\otimes R,G) $. Likewise, we let $ L_R $ and $ L_\mathfrak{R} $ be the reductions of the Laplacian of $ \Gamma $ modulo $ a $ and $ b $, respectively; we may think of $ L_R $ and $ L_\mathfrak{R} $ as $ n\times n $ matrices with coefficients in $ R $ and $ \mathfrak{R} $, respectively. Recall from \Cref{sec: Pairings} that $ \hcok(L) $ carries a symmetric pairing given by $ L $, which we denote by $ \p_L $. Likewise $ \hcok(L_R) $ and $ \hcok(L_{\mathfrak{R}}) $ are both equipped with symmetric pairings given by restricting the pairing on $ \hcok(L) $ to $ \hcok(L_R) $ and $ \hcok(L_{\mathfrak{R}}) $, respectively.
		
		Let $ M $ be a random $ n\times n $ symmetric matrix with coefficients in $ R $. Suppose the $ m_{ij} $ modulo $ b $ are distributed as $ (L_\mathfrak{R})_{ij} $ for $ i<j $, and suppose the $ m_{ii} $'s are distributed uniformly in $ \mathfrak{R} $, with all $ m_{ij} $ ($ i<j $) and $ m_{ii} $ independent. Set $ W=\mathfrak{R}^n $. Recall that we may view $ L_\mathfrak{R} $ as a linear map $ W^*\to W $, and let $ w_1,\ldots,w_n $ be our distinguished basis of $ W $ (i.e., the basis of $ W $ such that the strict upper-triangular entries of $ L_\mathfrak{R} $ with respect to the bases $ w_i $ and $ w_i^* $ of $ W $ and $ W^* $ are independent). Let $ F_0\in\Hom(W^n,\mathfrak{R}) $ be the map that sends each basis vector $ w_i\mapsto 1 $ (we see that $ F_0M $ is the row vector whose entries are the column sums of $ M $). Note that $ M $ and $ L_\mathfrak{R} $ do not have the same distribution, since the column sums of $ M $ can be anything and the column sums of $ L_\mathfrak{R} $ are zero. To fix this, we'll condition on $ F_0M=0 $ so that the conditioned distribution of $ M $ is the same as the distribution of $ L_\mathfrak{R} $. Let $ b_{ij}\in\mathfrak{R} $ for each $ 1\leq i<j\leq n $. We see that \[\Prob(F_0M=0\;|\;m_{ij}=b_{ij}\textrm{ for all }1\leq i<j\leq n)=b^{-n},\] so $ \Prob(F_0M=0)=b^{-n} $. Hence, given $ F_0M=0 $, we see that any choice of off-diagonal entries is equally likely in $ L_\mathfrak{R} $ as in $ M $. 
		
		For $ F\in\Hom(W,G) $, we have that \[\Prob(FL_\mathfrak{R}=0\mathrm{\ and\ }F^t(\p_{L_\mathfrak{R}})=\phG)=\Prob(FM=0\mathrm{\ and\ }F^t(\p_M)=\phG\;|\;F_0M=0),\] where $ \p_{L_\mathfrak{R}} $ and $ \p_M $ denote the symmetric pairings on $ \hcok(L_\mathfrak{R}) $ and $ \hcok(M) $, respectively. Also, \begin{align*}
			\Prob(FM=0&\mathrm{\ and\ }F^t(\p_M)=\phG\;|\;F_0M=0)\Prob(F_0M=0)\\
			&=\Prob(FM=0\mathrm{\ and\ }F^t(\p_M)=\phG\ \mathrm{and\ }F_0M=0).
		\end{align*} Let $ \tilde{F}\in\Hom(W,G\oplus \mathfrak{R}) $ be the sum of $ F $ and $ F_0 $.
		
		Let $ U\subset W $ denote the vectors whose coordinates sum to 0. In other words, let $ U=\{w\in W\;|\;F_0w=0\} $. Denote by $ \Sur_U(W,G) $ the maps from $ W\to G $ that are surjections when restricted to $ U $. 	
		We would like to estimate \begin{align*}
			\E(\#\Sur^*(S_\mathfrak{R},G))=\E(\Sur^*(U/\col(L_\mathfrak{R})),G)&=\sum_{F\in\Sur(U,G)}\Prob(FL_\mathfrak{R}=0\ \mathrm{and\ }F^t(\p_{\h{S_\mathfrak{R}}})=\p_{\hG}).
		\end{align*}
		Now, we have
		\begin{align*}
			\E(\#\Sur^*(S_\mathfrak{R},G))&=\sum_{F\in\Sur(U,G)}\Prob(FL_\mathfrak{R}=0\ \mathrm{and\ }F^t(\p_{\h{S_\mathfrak{R}}})=\p_{\hG})\\
			&=\frac{1}{|G|}\sum_{F\in\Sur_U(W,G)}\Prob(FL_\mathfrak{R}=0\mathrm{\ and\ }F^t(\p_{\h{S_\mathfrak{R}}})=\phG).
		\end{align*}
		
		Given $ F\in\Sur_U(W,G) $ such that $ FL_\mathfrak{R}=0 $, recall from \Cref{check pairing on cokernel} that $ F^t(\p_{\h{S_\mathfrak{R}}})=\phG $ if and only if $ F^t(\p_{L_\mathfrak{R}})=\phG $. In other words, for any such $ F $, we have that \[\Prob(FL_\mathfrak{R}=0\mathrm{\ and\ }F^t(\p_{\h{S_\mathfrak{R}}})=\phG)=\Prob(FL_\mathfrak{R}=0\mathrm{\ and\ }F^t(\p_{L_\mathfrak{R}})=\phG),\] and we can check whether the symmetric pairing on $ \h{S_\mathfrak{R}} $ pushes forward to $ \phG $ by checking whether the symmetric pairing on $ \hcok(L_\mathfrak{R}) $ pushes forward to $ \phG $. 
		
		Now, let $ V=R^n $, and let $ Z\subset V $ denote the submodule of vectors whose coordinates sum to 0 modulo $ b $, i.e., $ Z=\{v\in V\;|\;F_0v=0\mathrm{\ mod\ }b\} $, where we abuse notation and define $ F_0:V\to\mathfrak{R} $ to be the map taking each $ v_i\mapsto1 $ modulo $ b $. As before, set $ \Sur_Z(V,G) $ to be the set of maps $ V\to G $ that are surjections when restricted to $ Z $. Note that because $ \Sur(W,G)=\Sur(V,G) $, it follows that $ \Sur(U,G)=\Sur(Z,G) $. Therefore, \begin{align*}
			\E(\#\Sur^*(S_\mathfrak{R},G))&=\frac{1}{|G|}\sum_{F\in\Sur_U(W,G)}\Prob(FL_\mathfrak{R}=0\mathrm{\ and\ }F^t(\p_{L_\mathfrak{R}})=\phG)\\
			&=|G|\inverse b^n\sum_{F\in\Sur_U(W,G)}\Prob(\tilde{F}M=0\mathrm{\ and\ }F^t(\p_{M})=\phG)\\
			&=|G|\inverse b^n\sum_{F\in\Sur_Z(V,G)}\Prob(\tilde{F}M=0\mathrm{\ and\ }F^t(\p_{M})=\phG).
		\end{align*}
		
		Note that if $ F:V\to G $ is a surjection when restricted to $ Z $, then $ \tilde{F} $ is a surjection from $ V $ to $ G\oplus \mathfrak{R} $. Write the above as \[\E(\#\Sur^*(S_\mathfrak{R},G))=|G|\inverse b^n\sum_{F\in\Sur_Z(V,G)}\sum_{\substack{\mathrm{extensions}\ \p_{\hG\oplus\h{\mathfrak{R}}}\mathrm{\ of}\\ \phG\mathrm{\ to\ }\hG\oplus\h{\mathfrak{R}}}}\Prob(\tilde{F}M=0\mathrm{\ and\ }\tilde{F}^t(\p_M)=\p_{\hG\oplus\h{\mathfrak{R}}}),\] where by extensions $ \p_{\hG\oplus\h{\mathfrak{R}}} $ of $ \phG $ we really mean \emph{symmetric} extensions of $ \phG $ (we truncate for the sake of brevity). For the rest of the argument, when we refer to extensions of $ \phG $ we mean symmetric extensions of this pairing. If $ g_1,\ldots,g_r $ generate $ G $, and if $ t $ generates $ \mathfrak{R} $, then we see that symmetric extensions of $ \phG $ to $ \hG\oplus\h{\mathfrak{R}} $ are entirely determined by where the pairs $ (\h{g_i},\h{t}) $ and $ (\h{t},\h{t}) $ are sent for all $ 1\leq i\leq r $. Hence, in total, there are $ b|G| $ possible extensions of $ \phG $ to $ \hG\oplus\h{\mathfrak{R}} $. 
		
		We start by considering the part of the sum where $ \tilde{F} $ is not a code of distance $ \delta n $ for some $ \delta>0 $ that we will choose later. We allow $ K $ to change in each line, as long as it is a constant depending only on $ q,G,\delta $. Let $ \alpha=\max(q,1-q) $. Then 
		\begin{align*}
			\frac{b^n}{|G|}&\sum_{\substack{F\in\Sur_Z(V,G)\\\tilde{F}\mathrm{\ not\ code\ of\ distance\ }\delta n}}\sum_{\substack{\mathrm{extensions}\ \p_{\hG\oplus\h{\mathfrak{R}}}\mathrm{\ of}\\ \phG\mathrm{\ to\ }\hG\oplus\h{\mathfrak{R}}}}\Prob(\tilde{F}M=0\mathrm{\ and\ }\tilde{F}^t(\p_{M})=\p_{\hG\oplus\h{\mathfrak{R}}})\\
			&\leq\frac{b^n}{|G|}\sum_{\substack{D>1\\D|\# G}}\sum_{\substack{F\in\Sur_Z(V,G)\\\tilde{F}\ \mathrm{depth\ }D}}\sum_{\substack{\mathrm{extensions}\ \p_{\hG\oplus\h{\mathfrak{R}}}\mathrm{\ of}\\ \phG\mathrm{\ to\ }\hG\oplus\h{\mathfrak{R}}}}\Prob(\tilde{F}M=0\mathrm{\ and\ }\tilde{F}^t(\p_{M})=\p_{\hG\oplus\h{\mathfrak{R}}})\\
			&\leq \frac{b^{n}}{|G|}\sum_{\substack{D>1\\D|\# G}}\sum_{\substack{F\in\Sur_Z(V,G)\\\tilde{F}\ \mathrm{depth\ }D}}b|G|\Prob(\tilde{F}M=0)\\
			&\leq b^{n+1}\sum_{\substack{D>1\\D|\# G}}\#\{\tilde{F}\in\Hom(V,G\oplus \mathfrak{R})\ \mathrm{depth}\ D\;|\;\pi_2(v_i)=1\ \mathrm{for\ all\ }i\}\\
			&{}\quad\quad\quad\quad\quad\quad\times Ke^{-\alpha(1-\ell(D)\delta)n}(b|G|/D)^{-(1-\ell(D)\delta)n}\\
			&{}\leq b^{n+1}\sum_{\substack{D>1\\D|\# G}}K{{n}\choose{\lceil\ell(D)\delta n\rceil-1}}|G|^nD^{-n+\ell(D)\delta n}\times e^{-\alpha(1-\ell(D)\delta)n}(b|G|/D)^{-(1-\ell(D)\delta)n}\\
			&{}\leq K{{n}\choose{\lceil\ell(|G|)\delta n\rceil-1}}e^{-\alpha(1-\ell(|G|)\delta)n}(b|G|)^{\delta\ell(|G|)n}\\
			&{}\leq Ke^{-cn},
		\end{align*} where the third and fourth inequalities in the above follow from \Cref{sandpile depth,bound probability for depth}.
		The last inequality in the above follows from the fact that for any $ 0<c<\alpha $, we can choose $ \delta $ so small that \[{{n}\choose{\lceil\ell(|G|)\delta n\rceil-1}}e^{-\alpha(1-\ell(|G|)\delta)n}(b|G|)^{\delta\ell(|G|)n}\leq e^{-cn}.\]
		
		We also have that \begin{align*}
			\sum_{\substack{F\in\Sur_Z(V,G)\\\tilde{F}\mathrm{\ not\ code\ of\ distance\ }\delta n}}&|G|^{-n-1}\leq\sum_{\substack{D>1\\D|\# G}}\sum_{\substack{F\in\Sur_Z(V,G)\\\tilde{F}\ \mathrm{depth\ }D}}|G|^{-n-1}\\
			&\leq\sum_{\substack{D>1\\D|\# G}}K{{n}\choose{\lceil\ell(D)\delta n\rceil-1}}|G|^n|D|^{-n+\ell(D)\delta n}|G|^{-n-1}\\
			&\leq K{{n}\choose{\lceil\ell(|G|)\delta n\rceil-1}}2^{-n+\ell(|G|)\delta n}\\
			&\leq Ke^{-cn},
		\end{align*} where the last inequality holds because, for any $ 0<c<\log(2) $, we may choose $ \delta $ sufficiently small so that \[{{n}\choose{\lceil\ell(|G|)\delta n\rceil-1}}2^{-n+\ell(|G|)\delta n}\leq e^{-cn}.\]
		
		Similarly, we also have that \begin{align*}
			\sum_{F\in\Hom(V,G)\setminus\Sur_Z(V,G)}|G|^{-n-1}&\leq\sum_{\calh<G}\sum_{F\in\Sur(Z,\calh)}|G|^{-n}\leq\sum_{\calh<G}|\calh|^{n-1}|G|^{-n}\leq K2^{-n},
		\end{align*} where $ \calh<G $ indicates that $ \calh $ is a proper subgroup of $ G $. 
		
		Then, using \Cref{bound for codes} and \eqref{equivalent probabilities}, we see that \begin{align*}
			&\sum_{\substack{F\in\Sur_Z(V,G)\\\tilde{F}\mathrm{\ code\ of\ distance\ }\delta n}}\sum_{\substack{\mathrm{extensions}\ \p_{\hG\oplus\h{\mathfrak{R}}}\mathrm{\ of}\\ \phG\mathrm{\ to\ }\hG\oplus\h{\mathfrak{R}}}}\left|\Prob(\tilde{F}M=0\mathrm{\ and\ }\tilde{F}^t(\p_M)=\p_{\hG\oplus\h{\mathfrak{R}}})-(b|G|)^{-n-1}\right|\\
			&\qquad\leq\sum_{\substack{F\in\Sur_Z(V,G)\\\tilde{F}\mathrm{\ code\ of\ distance\ }\delta n}}\sum_{\substack{\mathrm{extensions}\ \p_{\hG\oplus\h{\mathfrak{R}}}\mathrm{\ of}\\ \phG\mathrm{\ to\ }\hG\oplus\h{\mathfrak{R}}}}Ke^{-cn}(b|G|)^{-n}\\
			&\qquad\leq Ke^{-cn}b^{-n}.
		\end{align*}
		
		Recall that the number of symmetric extensions of $ \phG $ to $ \hG\oplus\h{\mathfrak{R}} $ is $ b|G| $. Putting this all together, we have \begin{align*}
			\Bigg|&\frac{b^n}{|G|}\Bigg(\sum_{F\in\Sur_Z(V,G)}\sum_{\substack{\mathrm{extensions}\ \p_{\hG\oplus\h{\mathfrak{R}}}\mathrm{\ of}\\ \phG\mathrm{\ to\ }\hG\oplus\h{\mathfrak{R}}}}\Prob(\tilde{F}M=0\mathrm{\ and\ }\tilde{F}^t(\p_M)=\p_{\hG\oplus\h{\mathfrak{R}}})\Bigg)-|G|\inverse\Bigg|\\
			&\leq\left|\frac{b^n}{|G|}\sum_{\substack{F\in\Sur_Z(V,G)\\\tilde{F}\mathrm{\ not\ code\ of\ distance\ }\delta n}}\sum_{\substack{\mathrm{extensions}\ \p_{\hG\oplus\h{\mathfrak{R}}}\mathrm{\ of}\\ \phG\mathrm{\ to\ }\hG\oplus\h{\mathfrak{R}}}}\Prob(\tilde{F}M=0\mathrm{\ and\ }\tilde{F}^t(\p_M)=\p_{\hG\oplus\h{\mathfrak{R}}})\right|\\
			&\qquad+\frac{b^n}{|G|}\sum_{\substack{F\in\Sur_Z(V,G)\\\tilde{F}\mathrm{\ code\ of\ distance\ }\delta n}}\sum_{\substack{\mathrm{extensions}\ \p_{\hG\oplus\h{\mathfrak{R}}}\mathrm{\ of}\\ \phG\mathrm{\ to\ }\hG\oplus\h{\mathfrak{R}}}}\left|\Prob(\tilde{F}M=0\mathrm{\ and\ }\tilde{F}^t(\p_M)=\p_{\hG\oplus\h{\mathfrak{R}}})-(b|G|)^{-n-1}\right|\\
			&\qquad\qquad+\left|-|G|\inverse+\frac{b^n}{|G|}\sum_{\substack{F\in\Sur_Z(V,G)\\\tilde{F}\mathrm{\ code\ of\ distance\ }\delta n}}\sum_{\substack{\mathrm{extensions}\ \p_{\hG\oplus\h{\mathfrak{R}}}\mathrm{\ of}\\ \phG\mathrm{\ to\ }\hG\oplus\h{\mathfrak{R}}}}(b|G|)^{-n-1}\right|\\
			&\leq Ke^{-cn}+\sum_{\substack{F\in\Sur_Z(V,G)\\\tilde{F}\mathrm{\ not\ code\ of\ distance\ }\delta n}}|G|^{-n-1}+\sum_{F\in\Hom(V,G)\setminus\Sur_Z(V,G)}|G|^{-n-1}\\
			&\leq Ke^{-cn}.
		\end{align*} 
		
		Recalling that \[\E(\#\Sur^*(S_\mathfrak{R},G))= \frac{b^n}{|G|}\sum_{F\in\Sur_Z(V,G)}\sum_{\substack{\mathrm{extensions}\ \p_{\hG\oplus\h{\mathfrak{R}}}\\\mathrm{of\ }\phG\mathrm{\ to\ \hG\oplus\h{\mathfrak{R}}}}}\Prob(\tilde{F}M=0\mathrm{\ and\ }\tilde{F}^t(\p_M)=\p_{\hG\oplus\h{\mathfrak{R}}}),\] we may conclude the theorem. 
	\end{proof}
	\section{The moments determine the distribution}\label{moments determine distribution}
	In the following, we show that the moments we obtained in the previous section indeed determine the distributions of our random variables taking values in the category $ \calc $ of finite abelian groups whose duals are equipped with symmetric pairings. To do so, we combine several results from \cite{SW22} due to Sawin and Wood.
	
	Recall from \Cref{Introduction} that $ \calc $ is the category of finite abelian groups whose dual groups are equipped with symmetric pairings. For two objects $ (A,\p_{\h{A}}),(B,\p_{\h{B}})\in\calc $, recall that a morphism between these objects in $ \calc $ is a group homomorphism $ F:A\to B $ such that $ F^t(\p_{\h{A}})=\p_{\h{B}} $. To apply the moment methods of Sawin and Wood, we would like $ \calc $ to be a \emph{diamond category}.
	
	Before we give the definition of a diamond category, we first collect some notation from \cite{SW22}. Suppose $ \cald $ is a small category. Given an object $ G\in\cald $, a \emph{quotient} of $ G $ is the data of a object $ F\in\cald $ and an epimorphism $ \pi:G\to F $. Two quotients $ (F,\pi) $ and $ (F',\pi') $ of $ G $ are said to be \emph{isomorphic} if there exists an isomorphism $ \rho:F\to F' $ such that $ \pi'=\rho\circ\pi $. We may define a partial order on the set of isomorphism classes quotients of an object $ G\in\cald $, where $ F\leq H $ if there is some $ h:H\to F $ compatible with the epimorphisms from $ G $. Given a quotient $ F $ of $ G\in\cald $, let $ [F,G] $ be the partially ordered set of (isomorphism classes of) quotients of $ G $ that are greater than or equal to $ F $. Call an object \emph{minimal} if its only quotient is itself. 
	
	In our category of interest, we would like the set of quotients of an object with respect to this partial order to form a \emph{lattice}. A {lattice} is a poset such that any two elements $ x $ and $ y $ in the lattice have a least upper bound, or \emph{join}, and a greatest lower bound, or \emph{meet}. Denote the join of $ x $ and $ y $ (respectively, meet) by $ x\vee y $ (respectively, $ x\wedge y $). We say a lattice is \emph{modular} if $ a\leq b $ implies $ a\vee(x\wedge b)=(a\vee x)\wedge b $ for all $ x $ in the lattice. 
	
	For a set $ X $ of objects of $ \cald $, stable under isomorphism, $ X $ is said to be \emph{downward-closed} if every quotient of an object $ G\in X $ is also in $ X $. We say that $ X $ is \emph{join-closed} if for every $ G\in X $, given two quotients $ F $ and $ H $ of $ G $ such that $ F\vee H $ exists (in the poset of isomorphism classes of quotients of $ G $), we have $ F,H\in X $ implies $ F\vee H\in X $. Given a set $ X $ of isomorphism classes of objects in $ \cald $, the set \emph{generated} by $ X $ is the smallest downward-closed and join-closed subset of $ \cald $ containing $ X $ and every minimal object of $ \cald $. A \emph{level} $ D $ of $ \cald $ is a set of isomorphism classes of objects of $ \cald $ generated by a finite set. An epimorphism $ G\to F $ is called \emph{simple} if the interval $ [F,G] $ contains only $ F $ and $ G $. \begin{definition}
		A \emph{diamond category} is a small category $ \cald $ satisfying the following three properties. \begin{enumerate}
			\item For each $ G\in\cald $, the set of isomorphism classes of quotients of $ G $ form a finite modular lattice.
			\item Each object in $ \cald $ has finitely many automorphisms.
			\item For each level $ D $ of $ \cald $ and each object in the level $ D $, there are only finitely many elements of $ D $ with a simple epimorphism to $ G $.
		\end{enumerate}
	\end{definition}
	
	To see that $ \calc $ is a diamond category, we will apply the following lemma from \cite{SW22}. 
	
	\begin{lemma}[Lemma 6.22 of \cite{SW22}]\label{diamond category}
		Let $ \cald $ be a diamond category, and let $ \calg $ be a functor from $ \cald $ to the category of finite sets. Then the category $ (\cald,\calg) $ of pairs of an object $ G\in\cald $ together with an element $ s\in\calg(G) $, with morphisms $ (G,s_G)\to(H,s_H) $ given by morphisms $ f:G\to H $ in $ \cald $ such that $ \calg(f)(s_G)=s_H $, is a diamond category.
	\end{lemma}
	That $ \calc $ is a diamond category follows from applying \Cref{diamond category} to the category of finite abelian groups and the functor $ \Phi $ taking each group to the set of symmetric pairings on its dual group. Let $ \mathrm{FinAb} $ denote the category of finite abelian groups. The functor $ \Phi:\mathrm{FinAb}\to\mathrm{FinSet} $ takes a group homomorphism $ F:A\to B $ to the map of sets \[F^t:\{\mathrm{symmetric\ pairings\ on\ }\h{A}\}\to\{\mathrm{symmetric\ pairings\ on\ }\h{B}\}\] given by \[F^t(\p_{\h{A}})=\ip{F^t(-),F^t(-)}_{\h{A}}:\h{B}\otimes\h{B}\to\Q/\Z.\] We note that the category $ \calc $ is precisely the enriched category of pairs $ (\mathrm{FinAb},\Phi) $ as in \Cref{diamond category}.
	
	The following is the key result from \cite{SW22} that will allow us to prove the moments determine the distribution.
	\begin{lemma}[Theorem 1.6 of \cite{SW22}]\label{diamond moments det distr}
		Let $ \cald $ be a diamond category, and let $ D $ be a level in $ \cald $. For each $ G\in D $, let $ M_G $ be a real number such that $ (M_G)_G $ is well-behaved at $ D $.\footnote{The notion of well-behavedness at a level is not something we will precisely define. Roughly, well-behavedness captures whether or not the moments ``grow too quickly.'' In other words, if the moments in a diamond category are well-behaved at a level, then they will determine a unique measure on that level. The version of \Cref{diamond moments det distr} stated in \cite{SW22} is actually much stronger than the one stated above. However, for the sake of brevity, we use a weaker version of the theorem that is sufficient for our purposes.} Then if $ \mu_D^t$ and $\nu_D^t $ are two sequences of measures on $ D $ such that for each $ G\in D $, we have \[\lim_{t\to\infty}\int_{F\in D}|\mathrm{Epi}(F,G)|\mathrm{d}\mu_D^t=M_G=\lim_{t\to\infty}\int_{F\in D}|\mathrm{Epi}(F,G)|\mathrm{d}\nu_D^t,\] then $ \lim_{t\to\infty}\mu_D^t(\{F\}) $ and $ \lim_{t\to\infty}\nu_D^t(\{F\}) $ exist and are equal for any $ F\in D $. 
	\end{lemma}

	Note that if we set $ \cald=\calc $ in the above, then $ \mathrm{Epi}((A,\p_{\h{A}}),(B,\p_{\h{B}}))=\Sur^*(A,B) $ for $ (A,\p_{\h{A}}),(B,\p_{\h{B}})\in\calc $. Before showing that the $ (G,\phG) $-moments computed in \Cref{sec: Moments V} are well-behaved, we must first better understand the levels of $ \calc $ via the following lemma. \begin{lemma}
		Let $ C_b $ be the level of $ \calc $ generated by the set of objects whose underlying group is $ \Z/d\Z $ for $ d|b $ and whose pairing is any symmetric pairing on $ \Hom(\Z/d\Z,\Q/\Z) $. Then $ C_b $ is the following set of isomorphism classes of objects: finite abelian groups whose exponent divides $ b $ with all possible symmetric pairings on their dual groups. 
	\end{lemma}
	\begin{proof}
		We begin by noting that, given two groups $ A,B\in\finab $, both $ A $ and $ B $ are quotients of $ A\times B $ (we have $ B\simeq (A\times B)/(A\times\{0\}) $, for example). Moreover, $ A\times B $ is the join of $ A $ and $ B $ in the lattice of quotients of $ A\times B $. 
		
		Now, given any two objects $ (A,\p_{\h{A}}),(B,\p_{\h{B}})\in\calc $, let $ \p $ be a pairing on $ \h{A}\times\h{B}\simeq\Hom(A\times B,\Q/\Z) $ whose restriction to $ \h{A}\times\{0\} $ is $ \p_{\h{A}} $ and whose restriction to $ \{0\}\times\h{B} $ is $ \p_{\h{B}} $. For any such pairing $ \p $, we claim that $ (A\times B,\p) $ is the join of $ (A,\p_{\h{A}}) $ and $ (B,\p_{\h{B}}) $ in the lattice of quotients of $ (A\times B,\p) $. This follows from our remark in the previous paragraph and the fact that a quotient $ (G,\phG) $ of $ (A\times B,\p) $ is a quotient (of groups) $ \pi:A\times B\to G $ such that the transpose $ \pi^t:\h{G}\to\h{A}\times\h{B} $ pushes the pairing $ \p $ forward to the pairing $ \phG $.
		
		Let $ C $ be a level of $ \calc $, and suppose $ (A,\p_{\h{A}}),(B,\p_{\h{B}})\in C $ for any symmetric pairings $ \p_{\h{A}} $ and $ \p_{\h{B}} $ on $ \h{A} $ and $ \h{B} $, respectively. Given any symmetric pairing $ \phi $ on $ \h{A}\times\h{B} $, we note that $ (A\times B,\phi) $ is simply the join of $ (A,\phi|_{\h{A}}) $ and $ (B,\phi|_{\h{B}}) $, where $ \phi|_{\h{A}} $ denotes pairing on $ \h{A} $ given by the restriction of $ \phi $ to $ \h{A}\times\{0\} $ (similarly for $ \phi|_{\h{B}} $). Therefore, $ (A\times B,\phi)\in C $.

		Since any quotient of $ \Z/d\Z $ for $ d|b $ has exponent dividing $ b $, and since the join of $ \Z/d\Z $ with any other such group (in $ \finab $) has exponent dividing $ b $, we see that $ C_b $ only contains objects whose underlying group has exponent dividing $ b $. Now, any group $ G $ whose exponent divides $ b $ can be written as \[G=\bigoplus_{i=1}^r\Z/d_i\Z,\] where each $ d_i|b $. In other words, $ G $ is the join (as groups) of the $ \Z/d_i\Z $. By the above, the join-closedness of $ C_b $ forces $ C_b $ to contain $ G $ equipped with any possible symmetric pairing on its dual, since $ C_b $ is generated by $ \Z/d\Z $ with any symmetric pairing on its dual, where $ d $ is any divisor of $ b $. 
	\end{proof}
	
	Ultimately, we would like to apply \Cref{diamond moments det distr} to the category $ \calc $ at the level $ C_b $ for an arbitrary positive integer $ b $. Let $ D_b $ be the level in $ \finab $ generated by $ \Z/b\Z $, i.e., the level of finite abelian groups of exponent dividing $ b $. Note that $ D_b $ is equal to the level of $ \finab $ generated by $ \Z/d\Z $ for $ d|b $. To prove that the titular sequence of $ (G,\phG) $-moments computed in \Cref{sec: Moments V} are well-behaved, we reduce the question of well-behavedness at $ C_b $ to one of well-behavedness at $ D_b $ in $ \finab $ using the following lemma. \begin{lemma}[Proof of Lemma 6.23 of \cite{SW22}]\label{well-behaved: sum over pairings}
		Let $ \cald $ be a diamond category and $ \calg $ a functor from $ \cald $ to the category of finite sets. Let $ C $ be a level of the category of pairs, $ (\cald,\calg) $, and let $ D $ be the level of $ \cald $ generated by the underlying objects of a finite set of generators of $ C $. Let $ (M_H)_H\in\R^{C} $ be a set of moments. Then $ (M_H)_H $ is well-behaved at $ C $ if $ (\sum_{s\in\calg(G)}M_{(G,s)})_G\in\R^{D} $ is well-behaved at $ D $. 
	\end{lemma}
	When applied to $ \calc=(\mathrm{FinAb},\Phi) $, \Cref{well-behaved: sum over pairings} implies the following. For a sequence of moments $ (M_{(G,\phG)})_{(G,\phG)} $ indexed by elements of $ C_b $, if the sequence of moments indexed by $ G\in D_b $ given by \[\sum_{\substack{\phG\mathrm{\ a\ symmetric}\\\mathrm{pairing\ on\ }\hG}}M_{(G,\phG)}\] is well-behaved at $ D_b $ for $ \mathrm{FinAb} $, then the sequence $ (M_{(G,\phG)})_{(G,\phG)} $ is well-behaved at $ C_b $.
	
	Recall that for an object $ (G,\phG) \in C_b $, \Cref{sandpile moments} tells us that the $ (G,\phG) $-moment of the sandpile group of a random graph converges to $ |G|\inverse $ as $ n\to\infty $, where $ n $ is the number of vertices in the graph. Hence, we have reduced the problem to showing that the sequence in $ \R^{D_b} $ whose $ G $-entry is given by \[\frac{\#\{\mathrm{symmetric\ pairings\ on\ }\hG\}}{|G|}=\frac{|\Sym_2G|}{|G|}=|\wedge^2G|\] is well-behaved. 
	
	For any prime $ p $, let $ \calm_p $ denote the category of finite $ \Z_p $-modules, i.e., the category of finite abelian $ p $-groups. For any set of primes $ P $, let $ \finab_P $ be the product of categories $ \prod_{p\in P}\calm_p $. If $ P $ is the set of all primes, note that $ \finab=\finab_P $. \begin{lemma}[Lemma 6.16 of \cite{SW22}]\label{product of diamond is diamond}
		Let $ \cald_1 $ and $ \cald_2 $ be diamond categories. Then the product category $ \cald_1\times\cald_2 $, whose objects are ordered pairs of an object from $ \cald_1 $ and an object from $ \cald_2 $ and whose morphisms are ordered pairs of morphisms, is a diamond category.
	\end{lemma}
	\begin{lemma}[Lemma 6.17 of \cite{SW22}]\label{well-behaved product}
		Let $ \cald_1 $ and $ \cald_2 $ be diamond categories. Let $ (M_G^1)_{G\in\cald_1/\simeq} $ and $ (M_G^2)_{G\in\cald_2/\simeq} $ be tuples of nonnegative real numbers indexed by the respective isomorphism classes of $ \cald_1 $ and $ \cald_2 $. Define $ (M_{(G_1,G_2)})_{(G_1,G_2)\in\cald_1\times\cald_2/\simeq} $ by the formula $ M_{(G_1,G_2)}=M_{G_1}^1M_{G_2}^2 $. If  $ (M_G^1)_{G\in\cald_1/\simeq} $ and $ (M_G^2)_{G\in\cald_2/\simeq} $ are both well-behaved for $ \cald_1 $ and $ \cald_2 $, then $ (M_{(G_1,G_2)})_{(G_1,G_2)\in\cald_1\times\cald_2/\simeq} $ is well-behaved for the product category $ \cald_1\times\cald_2 $. 
	\end{lemma}
	Let $ P $ be a finite set of primes. For each $ p\in P $, consider a sequence of moments in $ \calm_p $, and form a sequence of moments in $ \finab_P $ by taking all products of the moments in the component categories $ \calm_p $. By \Cref{well-behaved product}, if each of the component sequences of moments are well-behaved, so is this sequence of moments in $ \finab_P $. Therefore, for the sequence of moments $ (|\wedge^2G|)_{G}\in\R^{\finab_P/\simeq} $, it suffices to check that the sequence $ (|\wedge^2G|)_G\in\R^{\calm_p/\simeq} $ is well-behaved for each $ p $, since \[\wedge^2G=\bigoplus_p\wedge^2(G_p).\] To do this, we simply apply the following lemma from \cite{SW22} to the category of finite $ \Z_p $-modules. \begin{lemma}[Lemma 6.9 of \cite{SW22}]\label{DVR well-behaved}
		Let $ R $ be a discrete valuation ring with finite residue field. For each isomorphism class $ N $ of finite $ R $-modules, let $ M_N $ be a real number. Then $ M_N $ is well-behaved if there exists some $ \epsilon>0 $ and $ c $ such that $ M_N\leq c|N|^{1-\epsilon}|\wedge^2N| $ for all finite $ R $-modules $ N $. 
	\end{lemma}
	
	Applying \Cref{DVR well-behaved} to $ \calm_p $, the category of finite $ \Z_p $-modules, it follows that the moments in $ \calm_p $ given by $ |\wedge^2G| $ for each $ G\in\calm_p $ are well-behaved. Hence, the corresponding moments in $ \finab_P $ are also well-behaved. To conclude the section, we repackage all of our above work in the following explicit theorem. \begin{theorem}\label{moments det distr II: explicit boogaloo}
		Let $ (X_n,\p_{\h{X_n}}) $ and $ (Y_n,\p_{\h{Y_n}}) $ be sequences of random variables taking values in the category of finitely generated abelian groups whose duals are equipped with symmetric pairings. Let $ b $ be a positive integer and $ C_b $ the set of isomorphism classes of finite abelian groups of exponent dividing $ b $ with symmetric pairings on their duals. Assume that for every $ (G,\phG)\in C_b $, \[\lim_{n\to\infty}\E(\#\Sur^*(X_n,G))=|G|\inverse\] and that the analogous statement holds for $ (Y_n,\p_{\h{Y_n}}) $. Then for every $ (G,\phG)\in C_b $ the limits $ \lim_{n\to\infty}\Prob((X_n\otimes\Z/b\Z,\p_{\h{X_n}})\simeq(G,\phG))$ and $ \lim_{n\to\infty}\Prob((Y_n\otimes\Z/b\Z,\p_{\h{Y_n}})\simeq(G,\phG))$ exist, and \[\lim_{n\to\infty}\Prob((X_n\otimes\Z/b\Z,\p_{\h{X_n}})\simeq(G,\phG))=\lim_{n\to\infty}\Prob((Y_n\otimes\Z/b\Z,\p_{\h{Y_n}})\simeq(G,\phG)).\] 
	\end{theorem}
	\begin{proof}
		Recall that $ C_b $ is the level of $ \calc $ generated by $ \Z/d\Z $ for $ d|b $ with every possible symmetric pairing on its dual. We have a sequence of $ (G,\phG) $-moments given by $ |G|\inverse $ for $ (G,\phG)\in C_b $, since for each $ (G,\phG)\in C_b $ we have \[\lim_{n\to\infty}\E(\#\Sur^*(X_n\otimes\Z/b\Z,G))=|G|\inverse.\] By \Cref{well-behaved: sum over pairings} and our previous discussion, to show that the moments are well-behaved at $ C_b $, it suffices to check well-behavedness at the level $ D_b $ of $ \finab $ (recall that $ D_b $ is the level of all finite abelian groups of exponent dividing $ b $). 
		
		Let $ P $ be the finite set of primes dividing $ b $, and consider $ \finab_P $. The level in $ \finab_P $ generated by $ \Z/b\Z $ is exactly the same as $ D_b $, so we may check well-behavedness at this level in $ \finab_P $. By \Cref{well-behaved product} and our work in the above, we see that the sequence of moments given by $ |\wedge^2G| $ for $ G\in D_b $ is well-behaved at $ D_b $, implying the sequence of moments $ (|G|\inverse)_{(G,\phG)\in C_b} $ is well-behaved at $ C_b $. Thus, we may apply \Cref{diamond moments det distr} and conclude the theorem.
	\end{proof}

	\section{Comparison to uniform random matrices and their pairings}\label{sec: Comparison to Haar}
	In the previous section, we showed that the moments we obtained for the sandpile groups of random graphs with their pairings determine a unique measure on certain levels of the category of finite abelian groups whose duals carry symmetric pairings. In order to determine the values of this measure at specific groups with pairings, we compare to the uniform case. In particular, \Cref{random matrix moment} implies that cokernels of uniform random matrices over $ \Z/a\Z $ with their pairings have the same moments as sandpile groups with their pairings. (Although the values of the moments of cokernels of uniform random matrices and their pairings are given by \Cref{random matrix moment}, we note that it is possible to compute these explicitly by adapting methods from \Cref{sec: Moments I} and the proof of Theorem 11 of \cite{CLK+}.) In \cite{CLK+}, Clancy, Kaplan, Leake, Payne, and Wood computed several of the aforementioned probabilities in the uniform case explicitly.	
	
	We will need the following result of \cite{CLK+}, which gives the (asymptotic) distribution of the cokernel of a uniform random symmetric matrix with its natural duality pairing. 
	\begin{lemma}[Theorem 2 of \cite{CLK+}]\label{Haar distribution}
		Let $ G $ be a finite abelian $ p $-group with a duality pairing $ \p_{\hG} $ on $ \hG $, and let $ M $ be a random $ n\times n $ symmetric matrix whose entries are distributed with respect to additive Haar measure on $ \Z_p $. Then the asymptotic probability (as $ n\to\infty $) that the cokernel of $ M $ with its duality pairing is isomorphic to $ (G,\phG) $ is \[\lim_{n\to\infty}\Prob((\cok(M),\p_{M})\simeq(G,\p_{\hG}))=\frac{\prod_{k\geq1}(1-p^{1-2k})}{|G||\Aut(G,\p_G)|}.\]
	\end{lemma}
	By extending the measure in the above result by 0, we can extend the above to give probabilities for when the pairing on $ \hG $ is any symmetric pairing. Notably, we do not require the pairing to be perfect. 
	
	\begin{corollary}[of \Cref{moments det distr II: explicit boogaloo}]\label{sandpile distribution is the same as random matrix}
		Let $ G $ be a finite abelian group of exponent dividing $ b $, and suppose $ G $ has a symmetric pairing on its dual $ \hG=\Hom(G,\Q/\Z) $, which we will denote using $ \phG $. Let $ 0<q<1 $, let $ \Gamma\in G(n,q) $ be a random graph, and let $ S $ be its sandpile group. Denote the canonical symmetric pairing on $ \h{S} $ by $ \p_{\h{S}} $. Also let $ H_n $ be a uniform random $ n\times n $ symmetric matrix with entries in $ \Z/b\Z $. Then \begin{align*}
			\lim_{n\to\infty}\Prob((S\otimes\Z/b\Z,\p_{\h{S}})\simeq(G,\phG))&=\lim_{n\to\infty}\Prob((\cok(H_n),\p_{H_n})\simeq(G,\phG)).
		\end{align*} In particular, if $ \phG $ is perfect, then \[\lim_{n\to\infty}\Prob((S\otimes\Z/b\Z,\p_{\h{S}})\simeq(G,\phG))=\frac{\prod_{p|b}\prod_{k\geq1}(1-p^{1-2k})}{|G||\Aut(G,\phG)|}.\] Otherwise, $ \lim_{n\to\infty}\Prob((S\otimes\Z/b\Z,\p_{\h{S}})\simeq(G,\phG))=0 $.
	\end{corollary}
	\begin{proof}
		That \begin{align*}
			\lim_{n\to\infty}\Prob((S\otimes\Z/b\Z,\p_{\h{S}})\simeq(G,\phG))&=\lim_{n\to\infty}\Prob((\cok(H_n),\p_{H_n})\simeq(G,\phG)).
		\end{align*} follows from applying \Cref{moments det distr II: explicit boogaloo}. In particular, we use the fact that both $ (S\otimes\Z/b\Z,\p_{\h{S}}) $ and $ (\cok(H_n),\p_{H_n}) $ have the same asymptotic $ (\calh,\p_{\h{\calh}}) $-moment for any object $ (\calh,\p_{\h{\calh}})\in C_b $ (this is a consequence of \Cref{sandpile moments}).
		
		The second part of the lemma follows from noting that everything factors over $ p|b $, allowing us to reduce to the case where $ G $ is a $ p $-group. 
	\end{proof}

	To prove \Cref{thm: distribution of gps and pairings}, we will turn \Cref{sandpile distribution is the same as random matrix} into a statement about the Sylow $ p $-subgroups of the sandpile group. In particular, we convert information about $ S\otimes\Z/b\Z $ for any $ b $ into information about the Sylow $ p $-subgroups of $ S $ for $ p|b $, and we translate statements about pairings on the dual of $ S\otimes\Z/b\Z $ to statements about pairings on the Sylow subgroups of $ S $. 
	\begin{remark}\label{rmk: sandpile finite}
		First, we will need to know that $ S $ is finite with asymptotic probability 1 as $ n\to\infty $. While this result is well-known, we offer an alternate proof below. Recall that $ S $ is finite if and only if $ \Gamma $ is connected; $ \Gamma $ is connected with probability 1 as $ n\to\infty $ for fixed $ q $ (this is a corollary of results proved by Erd\H{o}s and R\'{e}nyi in their 1960 paper \cite{ER}). Hence, as $ n\to\infty $, we know that $ S $ is finite with probability 1. Using Fatou's lemma and results from \cite{Wood} and \cite{CLK+}, we give an alternate proof of this fact. If we replace $ S $ in \Cref{sandpile is almost surely finite} by $ \cok(M) $, where $ M $ is a random matrix satisfying the hypotheses of \Cref{random matrix moment}, we get an analogous result for $ \cok(M) $. The result for $ \cok(M) $ is newer but also known---it is implied by Theorem 6.5 of \cite{CTV} and also by Theorem 5.1 of \cite{NW22}.
	\end{remark}
	\begin{lemma}\label{sandpile is almost surely finite}
		Let $ S $ be the sandpile group of a random graph on $ n $ vertices as in \Cref{sandpile moments}. As $ n\to\infty $, we have that $ \Prob(|S|<\infty)\to 1 $.
	\end{lemma}
	\begin{proof}
		Let $ P $ be any finite set of primes, and let $ \Z_P=\prod_{p\in P}\Z_p $. By Corollary 9.2 of \cite{Wood}, we have \[\lim_{n\to\infty}\Prob(S\otimes\Z_P\simeq G)=\frac{\#\{\mathrm{symmetric,\ perfect\ }\phi:G\otimes G\to\Q/\Z\}}{|G||\Aut(G)|}\prod_{p\in P}\prod_{k\geq1}(1-p^{1-2k})\] for any finite abelian $ P $-group $ G $.\footnote{The proof of Corollary 9.2 of \cite{Wood} works for our purposes, but note that it contains a slight mistake---throughout, ``Sylow $ P $-subgroup of $ S $'' should be replaced by $ S\otimes\Z_P $.} Proposition 7 of \cite{CLK+} tells us that the sum over all finite abelian $ P $-groups of the right-hand side of the above is 1. By Fatou's lemma, it follows that \[\lim_{n\to\infty}\Prob(|S|<\infty)=1,\] as desired. 
	\end{proof}

	\begin{corollary}\label{distribution of sandpile groups multiple primes}
		Let $ G $ be a finite abelian group equipped with a duality pairing $ \p_G $. Suppose $ \Gamma\in G(n,q) $ is a random graph with Laplacian $ L $ and sandpile group $ S $. Denote the canonical duality pairing on $ \tcok(L) $ by $ \p_{\tcok(L)} $. Suppose $ P $ is the finite set of prime divisors of $ |G| $. Let $ S_P $ denote the sum of the Sylow $ p $-subgroups of $ S $ for $ p\in P $ and $ \p_{\tcok(L),P} $ the restriction of $ \p_{\tcok(L)} $ to $ S_P $. Then \[\lim_{n\to\infty}\Prob((S_P,\p_{\tcok(L),P})\simeq(G,\p_G))=\frac{\prod_{p\in P}\prod_{k\geq1}(1-p^{1-2k})}{|G||\Aut(G,\p_G)|}.\] 
	\end{corollary}
	\begin{proof}	
		Write the exponent of $ G $ as $ \prod_{p\in P}p^{a_p} $, and let $ e=\prod_{p\in P}{p^{a_p+1}} $. Let $ \phG $ be the pairing on $ \hG $ induced by $ \p_G $. Recall that $ S_P\subset\tcok(L) $, so $ \p_{\tcok(L),P} $ is also perfect. Also recall that the dual of $ S\otimes\Z/e\Z $ comes equipped with a symmetric pairing, denoted $ \p_{\h{S}} $, given by pushing the pairing on $ \h{S} $ forward by the transpose of the natural quotient $ S\twoheadrightarrow S\otimes\Z/e\Z $. We have the following commutative diagram: \begin{center}
			\begin{tikzcd}
				S_P\arrow[r,hook] & \tcok(L)\arrow[r,hook] & S\arrow[r,two heads] & S\otimes\Z/e\Z\\
				\h{S_P}\arrow[u,leftrightarrow]&\htcok(L)\arrow[u,leftrightarrow]\arrow[l,two heads] & \h{S}\arrow[l,two heads] & \Hom(S\otimes\Z/e\Z,\Q/\Z)\arrow[l,hook],
			\end{tikzcd}
		\end{center}where the vertical maps are isomorphisms induced by the duality pairing $ \p_{\tcok(L)} $. 
		
		Assume $ S $ is finite. We will show that \[(S\otimes\Z/e\Z,\p_{\h{S}})\simeq(G,\phG)\textrm{ if and only if }(S_P,\p_{\tcok(L),P} )\simeq(G,\p_G).\] 
		
		To prove this, first note that $ S $ finite implies that $ S_P\simeq G $ if and only if $ S\otimes\Z/e\Z\simeq G $. Thus, it is sufficient to show that the data of the pairing $ \p_{\h{S}} $ on the dual of $ S\otimes\Z/e\Z $ is equivalent to the data of the pairing $ \p_{\tcok(L),P} $ on $ S_P $. Suppose $ S\otimes\Z/e\Z\simeq G\simeq S_P $. This forces the composition of maps in the first row of the above diagram to be an isomorphism between $ S_P $ and $ S\otimes\Z/e\Z $. The isomorphism $ S_P\simeq S\otimes\Z/e\Z $ gives us a duality pairing on $ S\otimes\Z/e\Z $, which in turn gives us a duality pairing on $ \Hom(S\otimes\Z/e\Z,\Q/\Z) $. We will prove that this induced pairing on $ \Hom(S\otimes\Z/e\Z,\Q/\Z) $ is just $ \p_{\h{S}} $. 
		
		Recall from \Cref{pairing same as Bosch-Lorenzini} that the pairing $ \p_{\h{S}} $ on $ \h{S} $ is induced by the duality pairing $ \p_{\tcok(L)} $ on $ \tcok(L) $. Moreover, recall that we may view the pairing $ \p_{\h{S}} $ on $ \Hom(S\otimes\Z/e\Z,\Q/\Z) $ as the restriction of $ \p_{\h{S}} $ to $ \Hom(S\otimes\Z/e\Z,\Q/\Z) $. Hence, to compute $ \p_{\h{S}} $ on $ \Hom(S\otimes\Z/e\Z,\Q/\Z) $, we may look at the corresponding elements in $ \tcok(L) $ and do the computation using $ \p_{\tcok(L)} $. Since the elements of $ \Hom(S\otimes\Z/e\Z,\Q/Z) $ all have order dividing $ e $, their images in $ \tcok(L) $ are $ P $-torsion. In other words, the image of $ \Hom(S\otimes\Z/e\Z,\Q/Z) $ in $ \tcok(L) $ lies in $ S_P $. It follows that $ \p_{\h{S}} $ is exactly the pairing on $ \Hom(S\otimes\Z/e\Z,\Q/\Z)\simeq\h{S_P} $ induced by $ \p_{\tcok(L_\Gamma),P} $, proving that $ (S\otimes\Z/e\Z,\p_{\h{S}})\simeq(G,\phG)\textrm{ if and only if }(S_P,\p_{\tcok(L),P} )\simeq(G,\p_G) $. 
		
		By \Cref{rmk: sandpile finite}, we know that $ S $ is asymptotically almost surely finite. Therefore, as $ n\to\infty $, we have that \[\Prob((S\otimes\Z/e\Z,\p_{\h{S}})\simeq(G,\phG))=\Prob((S_P,\p_{\tcok(L),P})\simeq(G,\p_G)),\] which, in conjunction with \Cref{sandpile distribution is the same as random matrix}, implies the result.
	\end{proof}

	\begin{remark}
		Because $ S_\Gamma $ is finite with probability 1 as $ n\to\infty $, in the limit, the pairing $ \p_{\tcok(L)} $ is simply the canonical duality pairing on $ S_\Gamma $. Hence, \Cref{distribution of sandpile groups multiple primes} can actually be viewed as a statement about the distribution of the Sylow $ P $-part of the sandpile group and its canonical duality pairing (more accurately, the restriction of this pairing to the Sylow $ P $-subgroup). 
		
		Moreover, all of the results in \Cref{sec: Comparison to Haar} hold if $ S $ is replaced with the cokernel of a random symmetric integral matrix satisfying the hypotheses of \Cref{random matrix moment}, where we use \Cref{random matrix moment} rather than \Cref{sandpile moments}.
	\end{remark}
	
	\section*{Acknowledgements}
	This research was conducted at Harvard University and the Duluth Summer Mathematics Research Program for Undergraduates at the University of Minnesota Duluth with support from Jane Street Capital, the National Security Agency, the National Science Foundation (grants 2140043 and 2052036), Harvard University, and the Harvard College Research Program. I am very grateful to Joe Gallian and Colin Defant for their support and for giving me the opportunity to return to Duluth. I would also like to extend special thanks to Nathan Kaplan, Gilyoung Cheong, and Dino Lorenzini for their insightful suggestions that greatly improved the quality of the paper. Finally, I am most deeply grateful to my research advisor, Melanie Matchett Wood, who, in addition to suggesting this problem, provided invaluable support and guidance throughout the research process.

	\bibliographystyle{acm}
	\bibliography{bib}
\end{document}